\documentclass[11pt,a4paper]{amsart}
\usepackage{a4wide}

\long\def\comment#1\endcomment{}
\newtheorem{theorem}{Theorem}[section]
\newtheorem*{theorem*}{Theorem}%[section]
\newtheorem{proposition}[theorem]{Proposition}
\newtheorem{lemma}[theorem]{Lemma}
\newtheorem*{conjecture}{Conjecture}
\theoremstyle{definition}
\newtheorem{defn}[theorem]{Definition}
\newtheorem*{example}{Example}
\theoremstyle{remark}
\newtheorem*{remark}{Remark}
\newtheorem*{remarks}{Remarks}
\newcommand{\wt}[1]{\widetilde{#1}}
\newcommand{\sier}[1]{{\cal O}_{#1}}
\newcommand{\C}{{\mathbb C}}
\renewcommand{\P}{{\mathbb P}}
\newcommand{\Z}{{\mathbb Z}}
\newcommand{\lra}{\longrightarrow}

\newcommand{\Ob}{\operatorname{Ob}}
\newcommand{\Bl}{\operatorname{Bl}}

\def\cal{\mathcal}
\def\be{{\beta}}
\def\ga{{\gamma}}
\def\al{{\alpha}}
\def\de{{\delta}}
\def\ep{{\varepsilon}}

\def\la{{\lambda}}
\def\si{{\sigma}}
\def\rh{{\rho}}

\newcommand{\cf}[1]{[#1_2,\dots,#1_{e-1}]}
\newcommand{\bcf}[1]{[\boldsymbol{#1}]}

\def\cir{{\circle*{0.23}}}
\def\cio{{\circle{0.23}}}
\def\lijn#1{{\line(1,#1){1}}}

\begin{document}
\title{The versal deformation of cyclic quotient singularities}

\author[Jan Stevens]{Jan Stevens}
\email {\sf stevens@chalmers.se}
\address
     {Matematiska Vetenskaper,
G\"oteborgs universitet,  SE 412 96 G\"oteborg, Sweden.} %%
\begin  {abstract}
We describe the versal deformation of two-dimensional
cyclic quotient singularities in terms of equations, 
following Arndt, Brohme and Hamm.
For the reduced components the equations are determined
by certain systems of dots in a triangle. The equations of
the versal deformation itself are governed by a different combinatorial
structure, involving  rooted trees.
\end{abstract}

\maketitle

%=============================================================================

One of the goals of singularity theory is to understand the
versal deformations of singularities. In general the base space itself is
a highly singular and complicated object. Computations for a whole class
of singularities are only possible in the presence of many symmetries.
A natural class of surface singularities to consider consists of
the affine toric
singularities. These are just the cyclic quotient singularities.
Their infinitesimal deformations were determined by Riemenschneider
\cite{ri}.
Explicit equations for the versal deformation are the result of a series
of PhD-theses. Arndt \cite{ar} gave  a recipe
to find equations of the base space. This was further studied by
Brohme \cite{sb}, who proposed explicit formulas. Their correctness
was finally proved by Hamm \cite{ha}. One of the objectives of this
paper is to describe these equations.

Unfortunately it is difficult to find the structure of the base space
from the equations. What one can do is to study the situation for
low embedding dimension $e$. On the basis of such computations
Arndt \cite{ar} conjectured  
that the number of irreducible
components should not exceed the Catalan number $C_{e-3}=
\frac{1}{e-2} \binom{2(e-3)}{ e-3}$. 
This conjecture was proved in \cite{js} using Koll\'ar and Shepherd-Barron's 
description \cite{ksb} of smoothing components as deformation spaces of 
certain partial resolutions. 
It was observed by Jan Christophersen that the components are related
to special ways of writing the equations of the singularity. In terms
of his continued fractions, representing zero,
these equations are given in \cite[\S 2]{ch},
and in terms of subdivisions of polygons in \cite[Sect.~6]{js}. A more
direct way of operating with the equations was found by Riemenschneider
\cite{br}. We use it, and the combinatorics behind it, 
in this paper to describe  the components. 

\comment
In this paper we describe equations
for the total space over each component, and the combinatorics behind
them. The fact that one  gets all components this way follows
from [loc.~cit.], but it should really be derived from the equations.
The  situation is complicated by the existence of embedded components.
Again on the basis of computations for low embedding dimension we 
give a conjectural description of the radical of the ideal.

Most of the issues can already be seen in the case of embedding
dimension 5 (the first case where the base space is singular), and
the cone over the rational normal curve of degree $4$ in particular.
A thorough understanding of it is important, yet this case is maybe
deceptively simple, and one has to go up in embedding dimension
to see what is going on.

It was observed by Jan Christophersen that the components are related
to special ways of writing the equations of the singularity, 
generalising the situation in embedding dimension 5. In terms
of his continued fractions, representing zero,
 these equations are given in \cite[\S 2]{ch},
and in terms of subdivisions of polygons in \cite[Sect.~6]{js}. A more
direct way of operating with the equations was found by Riemenschneider
\cite{br}. 
\endcomment
From the toric picture one finds immediately
some equations, by looking at the Newton boundary in the lattice
of monomials:
\[
z_{\ep-1}z_{\ep+1}=z_\ep^{a_\ep},\qquad 2\leq \ep\leq e-1\;.
\]
These form the bottom line of a pyramid of equations 
$z_{\de-1}z_{\ep+1}=p_{\de,\ep}$. In computing these higher equations
choices have to be made. We derive $p_{\de,\ep}$ from 
$p_{\de,\ep-1}$ and $p_{\de+1,\ep}$.
As $z_{\de-1}z_{\ep+1}=(z_{\de-1}z_{\ep})(z_{\de}z_{\ep+1}) /
(z_{\de}z_{\ep})$, we have two natural choices for $p_{\de,\ep}$:
\[
\frac { p_{\de,\ep-1} p_{\de+1,\ep} }{p_{\de+1,\ep-1}}
\qquad\text{or}\qquad
\frac { p_{\de,\ep-1} p_{\de+1,\ep} }{z_{\de}z_{\ep}}\;.
\]
We encode the choice by putting a white or black dot at place 
$(\de,\ep)$ in a triangle of dots. Only for certain systems of 
choices we can write down (in an easy way) enough deformations 
to fill a whole component. 
We call the corresponding triangles of dots sparse coloured triangles.
We prove that the number 
of sparse coloured triangles of given size is the Catalan number 
$C_{e-3}$. 

For the computation of the versal deformation one also starts
from the bottom line of the pyramid of equations. 
Due to the presence of deformation parameters, divisions which previously
were possible, now leave a remainder. We describe Arndt's formalism
to deal with these remainders. One introduces new symbols, which in
fact can be considered as new variables on the deformation space. Because they
are independent of the $a_\ep$, one obtains that the 
base spaces of different cyclic quotients with the same embedding dimension
are isomorphic up to multiplication by a smooth factor, provided all
$a_\ep$ are large enough. 
Also here, in writing the equations, some choices have to made.
A particular system of choices was proposed by Brohme.
To be able to handle the
terms in the formulas, one needs a combinatorial 
description of them. It turns out that the number of terms grows
rapidly, faster than the Catalan numbers, and a different
combinatorial structure is needed. Hamm \cite{ha} discovered how
rooted trees can be used.
We will describe the computation of the versal deformation
for embedding dimension 7 and then introduce Hamm's rooted trees,
and give the equations in general in terms of these trees. We also
describe the main steps in the proof that one really obtains the
versal deformation.

Now that the equations are known, it is time to use them. We make a start
here by showing that one recovers Arndt's equations for the 
versal deformation of the cones over rational normal curves (the case that
all $a_\ep=2$). Furthermore, we look at the reduced base space.
We start by looking at an example. 
We then define an ideal, using sparse coloured triangles, which has
the correct reduced components.
We do not
touch upon the embedded components, leaving this for further research.

As one will  see, notation becomes rather heavy, with many levels of
indices. Although \TeX\ allows almost anything, we have
tried to restrict it to a minimum. One has to admire Arndt's thesis
\cite{ar}, written on a typewriter. At that time, \TeX\ was available,
but J\"urgen had already purchased an electronic typewriter for his
Diplomarbeit. He decided to write the indices separately, 
diminish them with a photocopier and to glue them in the manuscript.

This paper is organised as follows. After a section introducing
cyclic quotients and their infinitesimal deformations, we treat 
the case of embedding dimension $5$ in detail. In Section 3 we 
define sparse coloured triangles and show how to describe the 
reduced components with them. In Section 4 we give the equations
for the total space of the versal deformation: we describe Arndt's
results, do the case of embedding dimension 6, and formulate and
sketch the proof of 
the general result in terms of Hamm's rooted trees. 
In the last Section we discuss the reduced base space.

%=============================================================================

\section{Cyclic quotient singularities}\label{cqs}
Let $G_{n,q}$ be the cyclic  subgroup of $Gl(2,\C)$, generated by
$\left( \begin{smallmatrix}\zeta_n & 0 \\ 0 & \zeta_n^q
\end{smallmatrix} \right)$,
where $\zeta_n$ is a primitive $n$-th root of unity and $q$ is coprime to $n$.
The group acts on $\C^2$ and on the polynomial ring $\C[u,v]$. 
% The action is abbreviated by $\frac1n(1,q)$.
The quotient $\C^2/G_{n,q}$ has a singularity at the origin, which is
called the cyclic quotient singularity $X_{n,q}$.
The quotient map is a map of affine toric varieties, given by the inclusion of
the standard lattice $\Z^2$ in the lattice $N=\Z^2+\Z\cdot\frac1n(1,q)$, with
as cone $\sigma$ the first quadrant. The dual lattice $M$ gives exactly the
invariant monomials: $\C[M\cap \sigma^\vee]=\C[u,v]^{G_{n,q}}$.
Generators of this ring are
\[
z_\ep=u^{i_\ep}v^{j_\ep}, \qquad \ep=1,\,\dots,e\;,
\]
where the numbers $i_\ep$, $j_\ep$ are determined by the continued
fraction expansion $n/(n-q)=\cf a$ in the following way:
\begin{eqnarray*}
i_e=0, &i_{e-1}=1, &i_{\ep+1}+i_{\ep-1}=a_\ep i_\ep\\
j_1=0, &j_2=1, &j_{\ep-1}+j_{\ep+1}=a_\ep j_\ep\;.
\end{eqnarray*}
We also  write $X\bcf a$ for $X_{n,q}$.

We exclude the case of the $A_k$ singularities and assume
that the embedding dimension $e$ is at least 4.
The equations for $X\bcf a$ can be given in quasi-determinantal format
 \cite{r2}:
\[
\left(
\begin{array}{ccccccc}
z_1 &              & z_2 & \dots &     z_{e-2} & &  z_{e-1} \\
\multicolumn{1}{c}{} &  z_2^{a_2-2} & \multicolumn{3}{c}{\dots}    &
z_{e-1}^{a_{e-1}-2} & \multicolumn{1}{c}{} \\
z_2 &              & z_3 & \dots &     z_{e-1}   &   &   z_{e} 
\end{array}
\right)
\]
We recall that the generalised minors of a quasi-determinant
\[\left(
\begin{array}{ccccccc}
f_1 &              & f_2 & \dots &    f_{k-1} & &  f_{k} \\
\multicolumn{1}{c}{} &  h_{1,2} & \multicolumn{3}{c}{\dots}    &
h_{k-1,k} & \multicolumn{1}{c}{} \\
g_1 &              & g_2 & \dots &    g_{k-1}   &    &   g_{k} 
\end{array}\right)
\]
are $f_ig_j - g_i(\prod_{\ep=i}^{j-1} h_{\ep,\ep+1})f_j$.

By perturbing the entries in the quasi-determinantal in the most
general way one obtains the equations for the Artin component, the
deformations which admit simultaneous resolution. We describe these
deformations more concretely, following the notation of
Brohme \cite{sb} (differing slightly from \cite{js} in that
the letters $s$ and $t$ are interchanged).
We first remark that in
general the first and the last extra term in a quasi-determinantal
can be written in the matrix: just take $g_1h_{1,2}$ as entry in the
lower left corner and $f_kh_{k-1,k}$ as entry in the upper right corner.
For a cyclic quotient this gives entries $z_2^{a_2-1}$ and 
$z_{e-1}^{a_{e-1}-1}$. 
We deform
\begin{equation}\label{artin}
\left(
\begin{array}{ccccccccc}
z_1 &   z_2 && z_3 & \dots &   z_{e-3}&&  z_{e-2}  & Z_{e-1}^{(a_{e-1}-1)}  \\
\multicolumn{2}{c}{} &  Z_3^{(a_3-2)} & \multicolumn{3}{c}{\dots}    &
Z_{e-2}^{(a_{e-2}-2)} & \multicolumn{2}{c}{} \\
 Z_2^{(a_2-1)} &z_3+t_3 & & z_4+t_4 & \dots & z_{e-2}+t_{e-2}&& z_{e-1} & z_{e} 
\end{array}
\right)
\end{equation}
where 
\[
Z_\ep^{(a_\ep-2)} =
z_\ep ^{a_\ep -2}+s_\ep^{(1)}z_\ep ^{a_\ep -3}+\dots+
s_\ep^{(a_\ep -2)}\;,
\]
and
\[
Z_\ep^{(a_\ep-1)} =
z_\ep ^{a_\ep -1}+s_\ep^{(1)}z_\ep ^{a_\ep -2}+\dots+
s_\ep^{(a_\ep -1)}\;.
\]
To obtain all infinitesimal deformations we add variables 
$s_\ep^{(a_\ep -1)}$ and write the perturbation 
\[
Z_\ep^{(a_\ep-2)} =
z_\ep ^{a_\ep -2}+s_\ep^{(1)}z_\ep ^{a_\ep -3}+\dots+
s_\ep^{(a_\ep -2)}+s_\ep^{(a_\ep -1)}z_\ep^{-1}\;,
\]
which gives the coordinates
$s_\ep^{(a)}$, $1\leq a\leq a_\ep-1$,
$\ep=2,\dots, e-1$ and $t_\ep$, $\ep=3,\dots, e-2$ on the vector space 
$T_{X\bcf a}^1$.
We note in particular the polynomial equations
\begin{equation}\label{infdef}
 z_{\ep-1}(z_{\ep+1}+t_{\ep+1}) =(z_{\ep}+t_{\ep})(z_\ep^{a_\ep-1}
 + s_\ep^{(1)}z_\ep^{a_\ep-2}+\dots + s_\ep^{(a_\ep-1)})\;.
\end{equation}
To avoid special cases we make this formula valid for all $\ep$
by introducing varaibles $t_2$,  $t_{e-1}$ and $t_e$, which we set to zero. 

%%%%%%%%%%%%%%%%%%%%%%%%%%%%%%%%%%%%%%%%%%%%%%%%%%%%%%%%%%%%%%%%%%%%%

\section{Embedding dimension 5}
\label{inbvijf}
The two components of the versal deformation of the cone over the rational
normal curve of degree four are related to the two different ways of
writing the equations. The largest, the Artin component, is obtained by
deforming the $2\times4$ matrix 
\[
\begin{pmatrix}
 z_1& z_2 &z_3&z_4 \\ z_2 &z_3&z_4   &z_5
\end{pmatrix}\;.
\]
The equations can also be written as $2\times2$ minors of the symmetric
$3\times3$ matrix
\[
\begin{pmatrix}
z_1& z_2 &z_3 \\ z_2 &z_3&z_4 \\ z_3&z_4   &z_5
\end{pmatrix}\;,
\]
and perturbing this matrix gives as total space the cone
over the Veronese embedding of $\P^2$. Riemenschneider observed that this
generalises to all cyclic quotients of embedding dimension $5$ \cite{ri}.
One can even give the equations as quasi-determinantals. For the Artin 
component we take as described above
\[\left(
\begin{array}{ccccc}
   z_1& z_2&  &z_3&z_4^{a_4-1} \\
   & & z_3^{a_3-2} & & \\
   z_2^{a_2-1} &z_3& &z_4   &z_5
\end{array}\right)
\]
and for the other component
\[\left(
\begin{array}{ccccc}
 z_1& &z_2 & &z_3^{a_3-1} \\
    & z_2^{a_2-2} & & &\\
 z_2& &z_3& &z_4 \\
   & & & z_4^{a_4-2} &  \\
 z_3^{a_3-1}& &z_4&  &z_5
\end{array}\right)\, .
\]
The meaning of the last symbol becomes clear if we write out the equations, 
which we have to do in order to generalise, 
as  for higher embedding dimension  only
the Artin component has such a nice determinantal description.

We  write a pyramid of equations. From the $2\times4$
quasi-determinantal we get
\[
\begin{matrix}
&z_1z_5=z_2^{a_2-1}z_3^{a_3-2}z_4^{a_4-1} & \\[2mm]
\multicolumn{3}{c}{z_1z_4=z_2^{(a_2-1)}z_3^{a_3-1}  \qquad 
z_2z_5= z_3^{a_3-1}z_4^{a_4-1}}   \\[2mm]
z_1z_3=z_2^{a_2} \;\; 
&z_2z_4=z_3^{a_3}  &
\;\;z_3z_5=z_4^{a_4}
\end{matrix}
\]
%\begin{align*}
%z_1z_5&=z_2^{a_2-1}z_3^{a_3-2}z_4^{a_4-1}  \\
%z_1z_4=z_2^{(a_2-1)}z_3^{a_3-1}   \;&\qquad
%z_2z_5= z_3^{a_3-1}z_4^{a_4-1}   \\
%z_1z_3=z_2^{a_2}  \qquad\qquad
%z_2z_4&=z_3^{a_3}   \qquad\qquad
%z_3z_5=z_4^{a_4}
%\end{align*}
and from the symmetric quasi-determinantal
%\begin{align*}
%z_1z_5&=z_2^{a_2-2}(z_3^{a_3-1})^2z_4^{a_4-2}  \\
%z_1z_4=z_2^{(a_2-1)}z_3^{a_3-1}   \;&\qquad
%z_2z_5= z_3^{a_3-1}z_4^{a_4-1}   \\
%z_1z_3=z_2^{a_2}  \qquad\qquad
%z_2z_4&=z_3^{a_3}   \qquad\qquad
%z_3z_5=z_4^{a_4}
%\end{align*}
\[
\begin{matrix}
&z_1z_5=z_2^{a_2-2}(z_3^{a_3-1})^2z_4^{a_4-2}  \\[2mm]
\multicolumn{3}{c}{z_1z_4=z_2^{(a_2-1)}z_3^{a_3-1}   \qquad
z_2z_5= z_3^{a_3-1}z_4^{a_4-1}}  \\[2mm]
z_1z_3=z_2^{a_2} & z_2z_4=z_3^{a_3}  &
z_3z_5=z_4^{a_4}
\end{matrix}
\]
The difference between these two systems of equations lies in the
top line. Observe that $
z_1z_5-z_2^{a_2-2}(z_3^{a_3-1})^2z_4^{a_4-2} 
=(z_1z_5-z_2^{a_2-1}z_3^{a_3-2}z_4^{a_4-1}) 
+ z_2^{a_2-2}z_3^{a_3-2}z_4^{a_4-2} (z_2z_4-z_3^{a_3} )
$.

To describe the deformation we introduce the following polynomials:
\[
Z_\ep^{(a_\ep-k_\ep)} =
z_\ep ^{a_\ep -k_\ep}+s_\ep^{(1)}z_\ep ^{a_\ep -k_\ep-1}+\dots+
s_\ep^{(a_\ep -k_\ep)}\;.
\]
By deforming the first set of equations we obtain
the Artin component: 
\[
\begin{matrix}
&z_1z_5=Z_2^{(a_2-1)}Z_3^{(a_3-2)}Z_4^{(a_4-1)}  \\[2mm]
\multicolumn{3}{c}{z_1z_4=Z_2^{(a_2-1)}Z_3^{(a_3-2)}z_3   \qquad
z_2z_5= (z_3+t_3)Z_3^{(a_3-2)}Z_4^{(a_4-1)}}   \\[2mm]
z_1(z_3+t_3)=z_2Z_2^{(a_2-1)} & 
z_2z_4=z_3Z_3^{(a_3-2)}(z_3+t_3)   &
z_3z_5=z_4Z_4^{(a_4-1)}
\end{matrix}
\]
The second set of equations leads to the other component:
\[
\begin{matrix}
&z_1z_5=Z_2^{(a_2-2)}(Z_3^{(a_3-1)})^2Z_4^{(a_4-2)}  \\[2mm]
\multicolumn{3}{c}{z_1z_4=z_2Z_2^{(a_2-2)}Z_3^{(a_3-1)}   \qquad
z_2z_5= Z_3^{(a_3-1)}Z_4^{(a_4-2)}z_4 }  \\[2mm]
z_1z_3=z_2Z_2^{(a_2-2)}z_2  &
z_2z_4=z_3Z_3^{(a_3-1)}  &
z_3z_5=z_4Z_4^{(a_4-2)}z_4
\end{matrix}
\]
Together these two components constitute the versal deformation.
They fit together to the deformation 
\begin{equation}\label{vdvijf}
\begin{matrix}
\multicolumn{3}{c}{ z_1z_5=Z_2^{(a_2-1)}Z_3^{(a_3-2)}Z_4^{(a_4-1)}+
  s_3^{(a_3-1)}Z_2^{(a_2-2)}Z_3^{(a_3-1)}Z_4^{(a_4-2)}}\\[2mm]
\multicolumn{3}{c}{z_1z_4=Z_2^{(a_2-1)}Z_3^{(a_3-1)}  \qquad
z_2z_5= \wt Z_3^{(a_3-1)}Z_4^{(a_4-1)} }  \\[2mm]
z_1(z_3+t_3)=z_2Z_2^{(a_2-1)}  \qquad&
z_2z_4=Z_3^{(a_3-1)}(z_3+t_3)   &
\qquad z_3z_5=z_4Z_4^{(a_4-1)}
\end{matrix}
\end{equation}
over the base space defined by the equations
$$
s_2^{(a_2-1)}s_3^{(a_3-1)} = t_3s_3^{(a_3-1)} =
s_3^{(a_3-1)}s_4^{(a_4-1)} =0\;.
$$
Here the factor $\wt Z_3^{(a_3-1)}$ is defined by the equation
$z_3\wt Z_3^{(a_3-1)}=(z_3+t_3) Z_3^{(a_3-1)}$, which is possible
because of the equation $t_3s_3^{(a_3-2)} =0$.

\begin{remark}
The equations for the versal deformation
restrict (by setting the deformation variables to zero)
to the equations of the singularity
in the preferred form for the Artin component. A choice has to be made, and
this one is sensible as the Artin component is the only component, which
exists for all cyclic quotients. Observe also that the right hand side
of the top equation in \eqref{vdvijf} is no longer a product. 
For the study of the non-Artin component, e.g., to determine adjacencies,
the adapted equations are much better suited.
We have therefore in general two tasks, to describe equations suited
for each reduced component separately, and to give equations for
the total versal deformation.
\end{remark}

\section{Equations for components}
The reduced components of the versal deformation are related
to ways of writing the equations of the singularity, as shown in
\cite{ch} and  \cite{js}. Here we give a description which first appeared
in \cite{br}.

We have to write the equations $z_{\de-1}z_{\ep+1}=p_{\de,\ep}$.
Motivated by the case of embedding dimension 5 we
want the right hand side of the equations  to be of the form
$p_{\de,\ep} = \prod_\be \bigl(z_\be^{a_\be-k_\be}\bigr)^{\al_\be}$. Here the
$k_\be$ and $\al_\be$ depend on $\ep-\de$ , but the formula should be
in some sense universal, it should hold for
all $a_\be$ large enough (for $a_\be-k_\be$ has to be non-negative).
The toric weight vectors $w_\be\in\Z^2$ of the variables
$z_\be$ should therefore satisfy the equations 
\[
w_\de+w_\ep=\sum \al_\be(a_\be-k_\be)w_\be\;,
\]
the same equations as encountered by Jan Christophersen (see the Introduction
of \cite{ch}).

We construct
a pyramid of equations $z_{\de-1}z_{\ep+1}=p_{\de,\ep}$,
where $2\leq\de\leq\ep\leq e-1$. We start from the base line containing
the $z_{\ep-1}z_{\ep+1}=z_\ep^{a_\ep}$, and construct the next lines
inductively. We have to make choices, which we encode in a subset 
$B(\underline\triangle)$ of the
set of pairs $(\de,\ep)$ with $1<\de<\ep< e$.
As $z_{\de-1}z_{\ep+1}=(z_{\de-1}z_{\ep})(z_{\de}z_{\ep+1}) /
(z_{\de}z_{\ep})$, we have two natural choices for $p_{\de,\ep}$:
we take
\[
P_{\de\ep}=
\begin{cases}
\frac { p_{\de,\ep-1} p_{\de+1,\ep} }{p_{\de+1,\ep-1}}\;,
&\text{if } (\de,\ep)\notin B(\underline\triangle), \\[2mm]
\frac { p_{\de,\ep-1} p_{\de+1,\ep} }{z_{\de}z_{\ep}}\;, &
\text{if } (\de,\ep)\in B(\underline\triangle).
\end{cases}
\]
We depict  our set by
a triangle $\underline\triangle$ of the type
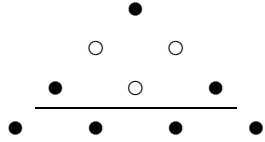
\newcommand{\wit}{\begin{picture}(6,6)(3,3) \put(3,3){\circle{4.5}} \end{picture}}
\newcommand{\zwart}{\begin{picture}(6,6)(3,3) \put(3,3){\circle*{5}}
\end{picture}}
\[
\begin{picture}(90,51)(-3,-3)
\put(0,0){\zwart}
\put(30,0){\zwart}
\put(60,0){\zwart}
\put(90,0){\zwart}
\put(7.5,7.5){\line(1,0){75}}
\put(15,15){\zwart}
\put(45,15){\wit}
\put(75,15){\zwart}
\put(30,30){\wit}
\put(60,30){\wit}
\put(45,45){\zwart}
\end{picture}
\]
The dots correspond to  equations  $z_{\de-1}z_{\ep+1}=p_{\de,\ep}$
above the base line in the pyramid of equations.
We colour a dot in $\underline\triangle$ 
black if the corresponding point $(\de,\ep)$ is an element of
$B(\underline\triangle)$.

As the second line of equations  always reads $p_{\ep-1,\ep}=
z_{\ep-1}^{a_{\ep-1}-1}z_\ep^{a_\ep-1}$, $2<\ep\leq e-1$, 
the  lowest line of the triangle
is coloured black. To characterise the coloured triangles, which give
good equations, it suffices to consider only triangles $\triangle$, obtained
by deleting this black line:
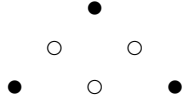
\[
\begin{picture}(75,36)(-3,12)
\put(15,15){\zwart}
\put(45,15){\wit}
\put(75,15){\zwart}
\put(30,30){\wit}
\put(60,30){\wit}
\put(45,45){\zwart}
\end{picture}
\]
The original triangle will be referred to as extended triangle.
We introduce some more terminology. A (broken)  line  $l_\ep$, in
both the triangle $\triangle$ and the extended triangle $\underline
\triangle$, is a line connecting all dots which have $\ep$ as one
of the coordinates:
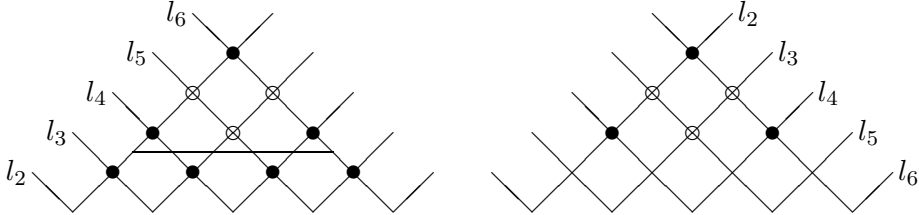
\[
\begin{picture}(150,80)(-30,-15)
\put(0,0){\zwart}
\put(30,0){\zwart}
\put(60,0){\zwart}
\put(90,0){\zwart}
\put(-15,-15){\line(1,1){75}}
\put(-15,-15){\line(-1,1){15}}
\put(15,-15){\line(1,1){60}}
\put(15,-15){\line(-1,1){30}}
\put(45,-15){\line(1,1){45}}
\put(45,-15){\line(-1,1){45}}
\put(75,-15){\line(1,1){30}}
\put(75,-15){\line(-1,1){60}}
\put(105,-15){\line(1,1){15}}
\put(105,-15){\line(-1,1){75}}
\put(7.5,7.5){\line(1,0){75}}
\put(15,15){\zwart}
\put(45,15){\wit}
\put(75,15){\zwart}
\put(30,30){\wit}
\put(60,30){\wit}
\put(45,45){\zwart}
\put(28,60){\makebox(0,0)[r]{$l_6$}}
\put(13,45){\makebox(0,0)[r]{$l_5$}}
\put(-2,30){\makebox(0,0)[r]{$l_4$}}
\put(-17,15){\makebox(0,0)[r]{$l_3$}}
\put(-32,0){\makebox(0,0)[r]{$l_2$}}
\end{picture}
\qquad
\begin{picture}(150,80)(-30,-15)
%\put(0,0){\zwart}
%\put(30,0){\zwart}
%\put(60,0){\zwart}
%\put(90,0){\zwart}
\put(-15,-15){\line(1,1){75}}
\put(-15,-15){\line(-1,1){15}}
\put(15,-15){\line(1,1){60}}
\put(15,-15){\line(-1,1){30}}
\put(45,-15){\line(1,1){45}}
\put(45,-15){\line(-1,1){45}}
\put(75,-15){\line(1,1){30}}
\put(75,-15){\line(-1,1){60}}
\put(105,-15){\line(1,1){15}}
\put(105,-15){\line(-1,1){75}}
%\put(7.5,7.5){\line(1,0){75}}
\put(15,15){\zwart}
\put(45,15){\wit}
\put(75,15){\zwart}
\put(30,30){\wit}
\put(60,30){\wit}
\put(45,45){\zwart}
\put(62,60){\makebox(0,0)[l]{$l_2$}}
\put(77,45){\makebox(0,0)[l]{$l_3$}}
\put(92,30){\makebox(0,0)[l]{$l_4$}}
\put(107,15){\makebox(0,0)[l]{$l_5$}}
\put(122,0){\makebox(0,0)[l]{$l_6$}}
\end{picture}
\]
If necessary we specify a triangle by the coordinates of its
vertex $(\de,\ep)$,  as $\triangle_{\de,\ep}$. The {\em height\/} 
of a triangle
$\triangle_{\de,\ep}$ is $\ep-\de-1$. This is the number of 
horizontal lines and also the number of dots on the base line.
A dot $(\al,\be)$ in a triangle $\triangle_{\de,\ep}$ determines
a sub-triangle $\triangle_{\al,\be}$, standing on the
same base line, of height $\be-\al-1$.

The crucial property for getting good equations is given in the following
definition.
\begin{defn} A coloured triangle $\triangle_{\de,\ep}$ is
{\em sparse}, if for it and for every sub-triangle $\triangle_{\al,\be}$
the number of black dots is at most the height of the triangle
with equality if and only if its vertex is black.
\end{defn}
Note that the example triangle above is sparse, whereas the following
triangle is not sparse.
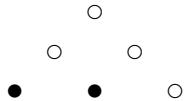
\[
\begin{picture}(75,36)(-3,12)
\put(15,15){\zwart}
\put(45,15){\zwart}
\put(75,15){\wit}
\put(30,30){\wit}
\put(60,30){\wit}
\put(45,45){\wit}
\end{picture}
\]

The relation $(\al,\ga)\preceq (\be,\de)$, if 
$\al\geq\be$ and $\ga\leq\de$ is a partial ordering. 
It means that $(\al,\ga)$ lies 
(as black or white dot) in the triangle $\triangle_{\be,\de}$.

\begin{lemma}\label{sparse}
If two black dots lie in the region on or above a given  line
$l_\ep$, then both of them  lie on $l_\ep$ or they are comparable in the
partial ordering $\preceq$.
\end{lemma}

\begin{proof}
Suppose on the contrary that the black dots $(\al,\ga)$ and $(\be,\de)$ 
on or above $l_\ep$ are not comparable in the partial ordering
and that at least one of them lies strictly above $l_\ep$.
We may assume that  $\ga\leq\de$. This implies that $\al<\be$.
The assumption that $(\al,\ga)$ lies on or above $l_\ep$ means that
$\al\leq\ep\leq\ga$ and likewise $\be\leq\ep\leq\de$. Furthermore in
one of these one has strict inequalities. Therefore $\be<\ga$.
The triangle $\triangle_{\al,\ga}$  contains exactly $\ga-\al-1$ black dots, 
the triangle $\triangle_{\be,\de}$  contains exactly $\de-\be-1$ black dots and
their intersection is the triangle $\triangle_{\be,\ga}$, which contains
at most $\ga-\be-1$ dots. So the triangle $\triangle_{\al,\de}$, which has
as vertex  the supremum $(\al,\de)$ of $(\al,\ga)$ and $(\be,\de)$ 
in the partial ordering, 
contains at least $(\ga-\al-1)+(\de-\be-1)-(\ga-\be-1)=\de-\al-1$ 
black dots other than its vertex, contradicting sparsity.
\end{proof}

\begin{theorem}
The number of sparse coloured triangles  of height $e-4$ 
is the Catalan number $C_{e-3}=\frac1{e-2}\binom{2(e-3)}{e-3}$.
\end{theorem}

\begin{proof}
Consider a sparse triangle $\triangle_{2,e-1}$ and let $(2,\be)$ be 
the highest black dot on the line $l_2$. There are no black dots above
the line  $l_\be$, for according to  Lemma \ref{sparse} the dot 
$(2,\beta)$ should lie in the
triangle of such a dot, implying that it lies on $l_2$, but
$(2,\be)$ is the highest black dot on that line. The triangle 
$\triangle_{\be,e-1}$ 
can be an arbitrary sparse triangle. The sparse triangle 
$\triangle_{3,\be}$ determines the colour of the remaining dots on the line
$l_2$: proceeding inductively downwards, the dot $(2,\ga)$ has to be black
if and only if there are exactly $\ga-\be$ black dots in the triangle
$\triangle_{2,\be}$, not lying in $\triangle_{2,\ga}$; as
$(2,\be)$ is black, there are at least $\ga-\be$ black dots in this complement.

This shows that the number $C_n$, $n=e-3$, of sparse  coloured triangles
of height $n-1$ satisfies Segner's
recursion formula for the Catalan numbers
\[
C_{n+1}=C_0C_n+C_1C_{n-1}+\cdots+C_{n-1}C_1+C_nC_0\;.
\]
For more information on the Catalan numbers, see \cite{st}.
\end{proof}

\begin{remark}
The Catalan number $C_{e-3}$ also counts the number of subdivisions
of an $(e-1)$-gon in triangles. An explicit bijection is as follows.
Mark, as in \cite{js}, a distinguished vertex and number the remaining
ones from 2 to $e-1$. If the vertices $\de$ and $\ep$ are joined
by a diagonal, then we colour the dot $(\de,\ep)$ black.
Conversely, given a triangle $\triangle_{2,e-1}$ we join
the vertices $\de$ and $\ep$ 
by a diagonal, if the dot $(\de,\ep)$ is black. By lemma \ref{sparse}
these diagonals do not intersect. We complete the subdivision
with diagonals through the distinguished vertex. Sometimes it is easier
to use subdivisions, but we will derive all facts we need directly
from the combinatorics of sparse triangles.
\end{remark}

To describe the equations we need the 
numbers $\cf k$ and  
$(\al_2,\dots,\al_{e-1})$. These are indeed the 
continued fractions
$\cf k$ representing zero \cite{ch} and the corresponding numbers  
%$(\al_2,\dots,\al_{e-1})$ 
satisfying $\al_{\ep-1}+\al_{\ep+1}=k_\ep\al_\ep$.
To define these numbers  out of a triangle we inductively give
weights to black dots. 

\begin{defn}\label{defak}
Let $\triangle_{2,e-1}$ be a sparse triangle. 
The weight $w_{\de,\ep}$ of a black dot
$(\de,\ep)$ is the sum of the weights of the dots lying in the sector above
it, increased by one: 
$w_{\de,\ep}= 1+ \sum_{\substack{\al<\de\\\be>\ep}} w_{\al,\be}$.
For $2\leq \ep\leq
e-1$ we define $\al_\ep$ as the sum of the weights of black dots
above the line $l_\ep$, increased by one:
$$
\al_\ep=1+\sum_{\substack{\al<\ep\\\be>\ep}} w_{\al,\be}\;.
$$
In particular, $\al_\ep=1$ if there are no black points above the line
$l_\ep$, so $\al_2=\al_{e-1}=1$.
\\
We set $k_\ep$ to be the number of black dots on the line $l_\ep$ in the
extended triangle  if $\al_\ep=1$ and this number minus $1$ otherwise.
\end{defn}

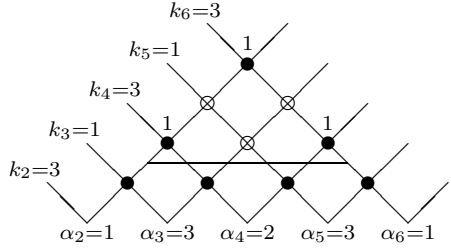
\begin{example}
\[
\begin{picture}(150,95)(-30,-25)
\put(0,0){\zwart}
\put(30,0){\zwart}
\put(60,0){\zwart}
\put(90,0){\zwart}
\put(-15,-15){\line(1,1){75}}
\put(-15,-15){\line(-1,1){15}}
\put(15,-15){\line(1,1){60}}
\put(15,-15){\line(-1,1){30}}
\put(45,-15){\line(1,1){45}}
\put(45,-15){\line(-1,1){45}}
\put(75,-15){\line(1,1){30}}
\put(75,-15){\line(-1,1){60}}
\put(105,-15){\line(1,1){15}}
\put(105,-15){\line(-1,1){75}}
\put(7.5,7.5){\line(1,0){75}}
\put(15,15){\zwart}
\put(45,15){\wit}
\put(75,15){\zwart}
\put(30,30){\wit}
\put(60,30){\wit}
\put(45,45){\zwart}
\put(15,20){\makebox(0,0)[b]{$\scriptstyle1$}}
\put(45,50){\makebox(0,0)[b]{$\scriptstyle1$}}
\put(75,20){\makebox(0,0)[b]{$\scriptstyle1$}}
\put(35,65){\makebox(0,0)[r]{$\scriptstyle k_6=3$}}
\put(20,50){\makebox(0,0)[r]{$\scriptstyle k_5=1$}}
\put(5,35){\makebox(0,0)[r]{$\scriptstyle k_4=3$}}
\put(-10,20){\makebox(0,0)[r]{$\scriptstyle k_3=1$}}
\put(-25,5){\makebox(0,0)[r]{$\scriptstyle k_2=3$}}
\put(105,-16){\makebox(0,0)[t]{$\scriptstyle\al_6=1$}}
\put(75,-16){\makebox(0,0)[t]{$\scriptstyle\al_5=3$}}
\put(45,-16){\makebox(0,0)[t]{$\scriptstyle\al_4=2$}}
\put(15,-16){\makebox(0,0)[t]{$\scriptstyle\al_3=3$}}
\put(-15,-16){\makebox(0,0)[t]{$\scriptstyle\al_2=1$}}
\end{picture}
\]
\end{example}

\begin{remarks}
1.\enspace We leave it as exercise to prove that the so defined numbers
$\al_\ep$ and $k_\ep$ satisfy $\al_{\ep-1}+\al_{\ep+1}=k_\ep\al_\ep$.
%\end{remark}
%\begin{remark}
\\ 2.\enspace
We note the following alternative way to compute the $\al_\ep$
\cite[Bemerkung 1.7]{sb}. As mentioned, $\al_\ep=1$ if there are
no black dots above the line $l_\ep$. For every other index $\ep$
there exist unique $\be<\ep<\ga$, such that the intersection
of the lines $l_\be$, $l_\ep$ and $l_\ga$ (in the 
extended triangle) consists only of black dots. Then 
$\al_\ep=\al_\be+\al_\ga$. In fact, this is the way the numbers
are determined from a subdivision of a polygon. We sketch a proof.
Let $(\be,\ep)$ be the highest black dot on the left half-line of
$l_\ep$, and $(\ep,\ga)$ the highest black dot on the right
half-line. Let $(\al,\de)$ be a black dot above the line $l_\ep$,
such that $\triangle_{\al,\de}$ contains no other black dots above $l_\ep$.
Then $(\be,\ga)\preceq(\al,\de)$, and if $(\be,\ga)\neq(\al,\de)$,
then $\triangle_{\al,\de}$ does not contain enough dots. So if
$\al_\ep\neq1$, then $(\be,\ga)$ is black. Black dots above the
lines $l_\be$, $l_\ep$ and $l_\ga$ can only lie in the sector 
with  $(\be,\ga)$  as lowest point. One now computes
$\al_\ep=\al_\be+\al_\ga$. 
\end{remarks}
We can now describe the equations belonging to a sparse triangle
$\triangle_{2,e-1}$. To avoid cumbersome notation  
we only give the formula for the highest equation $z_1z_e=p_{2,e-1}$,
but this implies by obvious changes the formula for each
equation $z_{\de-1}z_{\ep+1}=p_{\de,\ep}$, as such an
equation is determined by its
own sparse triangle $\triangle_{\de,\ep}$, giving its own $\al$ and $k$
values.  We will specify in the text 
from which triangle a specific $\al_\ep$
or $k_\ep$ is computed, but we do not include this information in the
notation.

\begin{proposition}
Let the triangle $\triangle_{2,e-1}$  determine the numbers
$\al_\ep$, $k_\ep$, according to Definition \ref{defak}. 
Forming the equations by taking
$p_{\de,\ep} =\frac { p_{\de,\ep-1} p_{\de+1,\ep} }{z_{\de}z_{\ep}}$ 
if the dot $(\de,\ep)$ is black and
$p_{\de,\ep} =\frac { p_{\de,\ep-1} p_{\de+1,\ep} }{p_{\de+1,\ep-1}}$
otherwise, leads to the highest equation
\[
z_1z_e=\prod_{\be=2}^{e-1} \bigl(z_\be^{a_\be-k_\be}\bigr)^{\al_\be}\;.
\]
\end{proposition}

\begin{proof}
We fix an index $\ep$ and look at the $z_\ep$-factor in $p_{2,e-1}$.
The proof proceeds by induction on $e$, i.e., on the height of the triangle. 
The base of the induction
is formed by the equations $z_{\ep-1}z_{\ep+1}=z_\ep^{a_\ep}$,
which correspond to empty extended triangles, with $\al_\ep=1$
and $k_\ep=0$.\\
Suppose that the formula is proved for all $p_{\de,\ep}$ with 
$\ep-\de<e-3$. There are two cases, depending on the colour
of the dot $(2,e-1)$.\\
{\em Case\/} 1:
the dot $(2,e-1)$ is white. Then we have to compare
$p_{2,e-1}$, $p_{2,e-2}$, $p_{3,e-1}$ and $p_{3,e-2}$ and the values
of $\al_\ep$ and $k_\ep$ computed from the corresponding triangles.
There are several sub-cases. In the first two we assume that
$\ep\neq2,e-1$. \\
1.a: Suppose both $(2,\ep)$ and $(\ep,e-1)$ are black. A black dot above the
line $l_\ep$ should contain both these points in its triangle, so it can
only be $(2,e-1)$, which however is assumed to be
uncoloured. Therefore $\al_\ep=1$ in all the relevant triangles. We have
that there are $k_\ep$ dots on the line $l_\ep$ in the extended triangle
${\underline\triangle}_{2,e-1}$,
$k_\ep-1$ in ${\underline\triangle}_{2,e-2}$ and 
${\underline\triangle}_{3,e-1}$,
and $k_\ep-2$ in ${\underline\triangle}_{3,e-2}$. So indeed
the $z_\ep$ factor in $p_{2,e-1}$ is equal to
$ z_\ep^{a_\ep-k_\ep+1}\cdot z_\ep^{a_\ep-k_\ep+1}/ z_\ep^{a_\ep-k_\ep+2}=
z_\ep^{a_\ep-k_\ep}$.\\
1.b: Otherwise the segments $(2,\ep)-(2,e-1)$ and $(\ep,e-1)-(2,e-1)$ cannot
both contain black dots. Suppose the first segment, on $l_2$, is empty (in
particular $(2,\ep)$ is white). Then the number of black dots
on the line $l_\ep$ is equal in both 
${\underline\triangle}_{2,e-1}$ and ${\underline\triangle}_{3,e-1}$,
being equal to $k_\ep$ or $k_\ep+1$ depending on the
value of $\al_\ep$, and one computes also the same value for $\al_\ep$.
Also ${\underline\triangle}_{2,e-2}$ and ${\underline\triangle}_{3,e-2}$
yield the same values $\al_\ep'$ and $k_\ep'$, so we get
$ \bigl(z_\ep^{a_\ep-k_\ep}\bigr)^{\al_\ep}
\cdot \bigl(z_\ep^{a_\ep-k_\ep'}\bigr)^{\al'_\ep}
/ \bigl(z_\ep^{a_\ep-k_\ep'}\bigr)^{\al'_\ep}=
\bigl(z_\ep^{a_\ep-k_\ep}\bigr)^{\al_\ep}$.\\
1.c: Suppose that $\ep=2$ or $\ep=e-1$. Consider the first case. The monomial
$z_2$ does not occur in $p_{3,e-1}$ and $p_{3,e-2}$. Always $\al_2=1$ and
the number of dots on $l_2$ is the same in both relevant triangles.\\
{\em Case\/} 2:
the dot $(2,e-1)$ is  black. We have to compare the values
of $\al_\ep$ and $k_\ep$ in  $p_{2,e-1}$, $p_{2,e-2}$
and $p_{3,e-1}$. Again we consider $\ep=2,e-1$ separately. \\
2.a:  Suppose both $(2,\ep)$ and $(\ep,e-1)$ are black. Then $(2,e-1)$ is the
only black dot above the line $l_\ep$, which makes $\al_\ep=2$, whereas
$\al_\ep=1$ in both smaller triangles. The number  of black dots on the line
$l_\ep$ is $k_e+1$ in ${\underline\triangle}_{2,e-1}$, and  $k_e$ 
in ${\underline\triangle}_{2,e-2}$ and ${\underline\triangle}_{3,e-1}$. 
So the $z_\ep$ factor in $p_{2,e-1}$ is equal to
$ z_\ep^{a_\ep-k_\ep}\cdot z_\ep^{a_\ep-k_\ep}=
\bigl(z_\ep^{a_\ep-k_\ep}\bigr)^2$.\\
2.b: Suppose  $(2,\ep)$ is black and $(\ep,e-1)$ not. Then all black points 
above
$l_\ep$ lie on $l_2$, all having weight 1, and there is at least one of them
between $(2,\ep)$ and $(2,e-1)$, as $(\ep,e-1)$ is not black. There are
$k_\ep+1$ points on $l_\ep$ in ${\underline\triangle}_{2,e-1}$
and ${\underline\triangle}_{2,e-2}$,  and $k_ \ep$ in
${\underline\triangle}_{3,e-1}$. The last triangle gives the value
$\al_\ep=1$ in $p_{3,e-1}$,
whereas $\al_\ep>1$ in the first two, and the value from
${\underline\triangle}_{2,e-1}$ is
one more than that from ${\underline\triangle}_{2,e-2}$.
So the $z_\ep$ factor in $p_{2,e-1}$ is equal to
$ \bigl(z_\ep^{a_\ep-k_\ep}\bigr)^{\al_\ep-1}\cdot z_\ep^{a_\ep-k_\ep}=
\bigl(z_\ep^{a_\ep-k_\ep}\bigr)^{\al_\ep}$. \\
2.c:  Suppose both $(2,\ep)$ and $(\ep,e-1)$ are white. The number of
dots on $l_\ep$ is the same in all three relevant triangles. In the largest
one $\al_\ep>1$, so to get the same value for $k_\ep$ we need that
$\al_\ep>1$ also in both other triangles. This means that the triangle
$\triangle_{3,e-2}$ has to have a dot above the line $l_\ep$.
Of the segments of
$l_2$ and $l_{e-1}$ above $l_\ep$ only one can contain black dots besides
the vertex. Suppose $(j,e-1)$ is the lowest dot of the segment on $l_{e-1}$.
The number of dots in $\triangle_{j,e-1}$ on or under
$l_\ep$ is at most $(e-1-\ep-2)+(\ep-j-1)=(e-1-j-3)$. As the triangle
contains exactly $e-1-j-2$ dots other than the vertex, there has to be
a dot above $l_\ep$, which does not lie on $l_{e-1}$ due to the choice of
$(j,e-1)$. To compute $\al_\ep$ we have to look at the weights. With the
convention that points above a triangle have weight 0, the inductive formula
holds for points in a sub-triangle with summation over all points 
in the sector of the big triangle.
We show by induction that for all points except the vertex $(2,e-1)$ the
weight $w_{i,j}$ computed from the big triangle, equals the sum of the
weights $w'_{i,j}$ from $\triangle_{2,e-2}$ and
$w''_{i,j}$ from $\triangle_{3,e-1}$. Indeed,
$w_{i,j}=1+\sum w_{k,l}= 1+1+\sum_{(k,l)\neq(2,e-1)}w_{k,l}=
1+ \sum w'_{k,l}+1+\sum w''_{k,l}=w'_{i,j}+w''_{i,j}$. The same
computation shows that also the values of $\al_\ep$  add. So indeed the
$z_\ep$ factor in $p_{2,e-1}$ is the product of those in $p_{2,e-2}$
and $p_{3,e-1}$.\\
2.d: If $\ep=2$, then $\al_2=1$ and the number of points on $l_2$ in 
${\underline\triangle}_{2,e-1}$ is $k_2$, whereas it is $k_2-1$ in 
${\underline\triangle}_{2,e-2}$.
The $z_2$ factor in $p_{2,e-1}$ is $z_2^{a_2-k_2+1}/z_2=z_2^{a_2-k_2}$.
The case $\ep=e-1$ is similar.
\end{proof}

As in the case of embedding dimension 5 we can now deform.
We perturb each term $z_\ep^{a_\ep-k_\ep}$ to
\[
Z_\ep^{(a_\ep-k_\ep)} =
z_\ep ^{a_\ep -k_\ep}+\wt s_\ep^{(1)}z_\ep ^{a_\ep -k_\ep-1}+\dots+
\wt s_\ep^{(a_\ep -k_\ep)}\;.
\]
Here we write $\wt s_\ep^{(j)}$, as these variables are not quite the same
as the coordinates  $s_\ep^{(j)}$ on $T^1$, specified by the equations 
\eqref{infdef}.
The relation is the following: if $\al_\ep>1$, we set $t_\ep=0$, and
$\wt s_\ep^{(j)}=s_\ep^{(j)}$, but if $\al_\ep=1$
one has
\begin{multline*}
(z_{\ep}+t_{\ep})^{\al_{\ep-1}}
(z_\ep ^{a_\ep -k_\ep}+\wt s_\ep^{(1)}z_\ep ^{a_\ep -k_\ep-1}+\dots+
\wt s_\ep^{(a_\ep -k_\ep)})z_\ep^{\al_{\ep+1}} \\
 {}=(z_{\ep}+t_{\ep})(z_\ep^{a_\ep-1}
 + s_\ep^{(1)}z_\ep^{a_\ep-2}+\dots + s_\ep^{(a_\ep-1)})\;.
\end{multline*}
This formula makes sense, as  $\al_\ep k_\ep=\al_{\ep-1}+\al_{\ep+1}$,
so for $\al_\ep=1$ one has $k_\ep=\al_{\ep-1}+\al_{\ep+1}$.

\begin{proposition}
Let $\triangle_{2,e-1}$ be a sparse triangle.
Put $t_\ep=0$, if $\al_\ep>1$. Now form the equations
$z_{\de-1}(z_{\ep+1}+t_{\ep+1})=P_{\de,\ep}$,
starting from 
\[
P_{\ep,\ep}= (z_{\ep}+t_{\ep})^{\al_{\ep-1}}
(z_\ep ^{a_\ep -k_\ep}+\wt s_\ep^{(1)}z_\ep ^{a_\ep -k_\ep-1}+\dots+
\wt s_\ep^{(a_\ep -k_\ep)})z_\ep^{\al_{\ep+1}} \;,
\]
if $\al_\ep=1$ and 
\[
P_{\ep,\ep}= 
(z_\ep ^{a_\ep -k_\ep}+ s_\ep^{(1)}z_\ep ^{a_\ep -k_\ep-1}+\dots+
 s_\ep^{(a_\ep -k_\ep)})z_\ep^{k_{\ep}} 
\]
otherwise. Take 
$ P_{\de,\ep} =\frac {  P_{\de,\ep-1}  P_{\de+1,\ep} }
   {z_{\de}(z_{\ep}+t_\ep)}$ 
if the dot $(\de,\ep)$ is black and
$ P_{\de,\ep} =\frac {  P_{\de,\ep-1}  P_{\de+1,\ep}}
{ P_{\de+1,\ep-1}}$
otherwise. This gives the highest equation
\[
z_1z_e=\prod_{\be=2}^{e-1} \bigl(Z_\be^{(a_\be-k_\be)}\bigr)^{\al_\be}\;.
\]
These equations define a flat deformation of the cyclic quotient 
singularity $X\bcf a$.
\end{proposition}

The flatness is proved explicitly in \cite[2.1.2]{ch} and \cite[2.2]{sb}.
It is of course a consequence of the inductive definition of the
polynomials $ P_{\de,\ep}$.

In fact, one gets in this way exactly all reduced components of the  
versal deformation. This was proved in
\cite{js} using Koll\'ar and Shepherd-Barron's 
description \cite{ksb} of smoothing components as deformation spaces of 
certain partial resolutions.  A more elementary (but not easier)
approach would be to use the equations for the base space of the versal
deformation, which we describe in the next section.

%=============================================================================

\section{Versal deformation}
In this section we derive the equations for the versal deformation.
We have to write the pyramid of equations, as in the example of
embedding dimension five. The base line 
consist of the equations (\ref{infdef}).
These equations are lacking in symmetry: when introducing the 
deformation variables $t_\ep$, say in the quasi-determinantal, there
is a choice of writing them in the upper or the lower row. Arndt \cite{ar}
formally symmetrises by setting $y_\ep=z_\ep+t_\ep$.
We go one step further and replace $t_\ep$ by two deformation
variables. This makes that our deformation is versal, but no longer 
miniversal. Furthermore, there is no $t_2$ and $t_{e-1}$, but in order to
avoid special cases, we allow the index $\ep$ in $t_\ep$ to take the
values $2$ and $e-1$. 
%
%The form of the formulas $z_{\de-1}z_{\ep+1}=
%p_{\de,\ep}$ depends only on the difference $\ep-\de$. The top formula
%of an pyramid of equations occurs also somewhere in the middle
%in  pyramids for higher embedding dimension.

We start from the equations $z_{\ep-1}z_{\ep+1} =z_\ep^{a_\ep}$, which
we deform into 
\begin{equation}\label{infdeg}
 (z_{\ep-1}-l_{\ep-1})(z_{\ep+1}-r_{\ep+1}) =z_\ep^{a_\ep}
 + \si_\ep^{(1)}z_\ep^{a_\ep-1}+\dots + \si_\ep^{(a_\ep)}\;.
\end{equation}
We abbreviate $z_\ep-r_\ep=R_\ep$ and $z_\ep-l_\ep=L_\ep$. The minus
sign is introduced to simplify the conditions for divisibility by 
$R_\ep$ or $L_\ep$, which will be the main ingredient in our
description of the base space. We write the equation (\ref{infdeg})
shortly as
\[
L_{\ep-1}R_{\ep+1}=Z_\ep^{(00)}\;.
\]
As written, we do not even get an infinitesimal deformation:
one needs $\si_\ep^{(a_\ep)}\equiv0$ modulo the  square of the
maximal ideal of the parameter space. Comparison with  
equation (\ref{infdef}) shows that $Z_\ep^{(00)}(z_\ep)$ 
(we use this notation to emphasise that we consider $Z_\ep^{(00)}$
as polynomial in $z_\ep$) has to be divisible by
$R_\ep$, i.e, $Z_\ep^{(00)}(r_\ep)=0$. This gives an equation
with non-vanishing linear part.
We could as well require divisibility by $L_\ep$. This gives an
equation $Z_\ep^{(00)}(l_\ep)=0$ with the same linear part.

In fact, we shall assume both conditions, $Z_\ep^{(00)}(r_\ep)=0$
and $Z_\ep^{(00)}(l_\ep)=0$. This yields then one equation
with non-vanishing linear part, and one equation, not involving
$\si_\ep^{(a_\ep)}$ at all, which factorises:
\[
Z_\ep^{(00)}(l_\ep)-Z_\ep^{(00)}(r_\ep)=(l_\ep-r_\ep)\si_\ep^{(11)}=0\;.
\]
This formula defines $ \si_\ep^{(11)}$, which should not be confused
with one of the variables in equation (\ref{infdeg}).
Those variables do not play a prominent role in the computations to
come. They are important for the momodromy covering of the
versal deformation, as noted by Riemenschneider and studied by
Brohme \cite{sb}. There is a large covering, which induces the
monodromy covering of each reduced component, obtained
by considering the $\si_\ep^{(i)}$ as elementary symmetric functions
in new variables. For details we refer to \cite{sb}.

We have to give the other equations. They will have the form
\[
L_{\de-1}R_{\ep+1}= P_{\de,\ep}\:.
\]
The polynomials $P_{\de,\ep}$ will be well defined modulo the
ideal $J$, generated by the equations of the base space. To 
describe them we 
perform division with remainder. 
\begin{defn}\label{deflr}
We inductively define
polynomials $Z_\ep^{(ij)}$ in the variable $z_\ep$, 
starting from $Z_\ep^{(00)}=z_\ep^{a_\ep}
 + \si_\ep^{(1)}z_\ep^{a_\ep-1}+\dots + \si_\ep^{(a_\ep)}$, by division
by $L_\ep$
\begin{equation}\label{links}
Z_\ep^{(ij)}=L_\ep Z_\ep^{(i+1,j)}+\si_\ep^{(i+1,j)}\;,
\end{equation}
and by $R_\ep$
\begin{equation}\label{rechts}
Z_\ep^{(ij)}= Z_\ep^{(i,j+1)}R_\ep+\si_\ep^{(i,j+1)}\;.
\end{equation}
\end{defn}
Note that $\si_\ep^{(i+1,j)}=Z_\ep^{(ij)}(l_\ep)$, and 
$\si_\ep^{(i,j+1)}=Z_\ep^{(ij)}(r_\ep)$.
From the equations (\ref{links}) or (\ref{rechts}) we obtain by 
substituting that
\begin{equation}\label{linrel}
\si_\ep^{(i+1,j)}-\si_\ep^{(i,j+1)}=
(l_\ep-r_\ep)\si_\ep^{(i+1,j+1)}\;.
\end{equation}
The condition $Z_\ep^{(00)}(l_\ep)=Z_\ep^{(00)}(r_\ep)=0$ translates into
$\si_\ep^{(10)}=\si_\ep^{(01)}=0$ and we can write
\[
Z_\ep^{(00)}=L_\ep Z_\ep^{(10)}= Z_\ep^{(01)}R_\ep\;.
\]

The next line in the pyramid of equations can now be computed:
\[
L_{\ep-2}R_{\ep+1}=
\frac{(L_{\ep-2}R_{\ep-1})(L_{\ep-1}R_{\ep+1})}{L_{\ep-1}R_{\ep}}= 
\frac{Z_{\ep-1}^{(00)}Z_{\ep}^{(00)}}{L_{\ep-1}R_{\ep}}=
Z_{\ep-1}^{(10)}Z_{\ep}^{(01)}\;.
\]

For the higher lines we do not quite
proceed as before, when describing the components. Computing with
$L_{\de-1}R_{\ep+1}=(L_{\de-1}R_{\ep})(L_{\de}R_{\ep+1})
/(L_{\de}R_{\ep})$ would be too complicated. Instead
we take the asymmetric approach $L_{\de-1}R_{\ep+1}=(L_{\de-1}R_{\ep})
(L_{\ep-1}R_{\ep+1})/(L_{\ep-1}R_{\ep})=P_{\de,\ep-1}
Z_\ep^{(01)}/L_{\ep-1}$.

We do the next step:
\[
L_{\ep-3}R_{\ep+1}=\frac{P_{\ep-2,\ep-1}Z_{\ep}^{(01)}}{L_{\ep-1}}
=\frac{Z_{\ep}^{(10)}Z_{\ep-1}^{(01)}Z_{\ep}^{(01)}}{L_{\ep-1}}\;,
\]
where we now have to use the division with remainder
$Z_{\ep-1}^{(01)}=L_{\ep-1} Z_{\ep-1}^{(11)}+\si_{\ep-1}^{(11)}$
of equation (\ref{links}) to get
\[
L_{\ep-3}R_{\ep+1}=
Z_{\ep-2}^{(10)}Z_{\ep-1}^{(11)}Z_{\ep}^{(01)}+
\frac{Z_{\ep-2}^{(10)}\si_{\ep-1}^{(11)}Z_{\ep}^{(01)}}{L_{\ep-1}}\;.
\]
This is not the final formula, as we can pull out a factor $L_{\ep-2}$
from $Z_{\ep-2}^{(10)}$ and $R_{\ep}$ from $Z_{\ep}^{(01)}$ by division
with remainder. Doing this successively and then using 
$L_{\ep-2}R_{\ep}=L_{\ep-1} Z_{\ep-1}^{(10)}$ gives us
\begin{multline}\label{mindrie}
L_{\ep-3}R_{\ep+1}=
Z_{\ep-2}^{(10)}Z_{\ep-1}^{(11)}Z_{\ep}^{(01)}+
\frac{L_{\ep-2}Z_{\ep-2}^{(20)}\si_{\ep-1}^{(11)}Z_{\ep}^{(01)}}{L_{\ep-1}}+
\frac{\si_{\ep-2}^{(20)}\si_{\ep-1}^{(11)}Z_{\ep1}^{(01)}}{L_{\ep-1}}\\
{}=
Z_{\ep-2}^{(10)}Z_{\ep-1}^{(11)}Z_{\ep}^{(01)}+
\frac{L_{\ep-2}Z_{\ep-1}^{(20)}\si_{\ep-1}^{(11)}
         Z_{\ep}^{(02)}R_{\ep}}{L_{\ep-1}}+
\frac{\si_{\ep-2}^{(20)}\si_{\ep-1}^{(11)}Z_{\ep}^{(01)}}{L_{\ep-1}}+
\frac{L_{\ep-2}Z_{\ep-2}^{(20)}\si_{\ep-1}^{(11)}\si_{\ep}^{(02)}}{L_{\ep-1}}
\\
{}=
Z_{\ep-2}^{(10)}Z_{\ep-1}^{(11)}Z_{\ep}^{(01)}+
Z_{\ep-2}^{(20)}\si_{\ep-1}^{(11)}Z_{\ep-1}^{(10)}Z_{\ep}^{(02)}+
\frac{\si_{\ep-2}^{(20)}\si_{\ep-1}^{(11)}Z_{\ep}^{(01)}}{L_{\ep-1}}+
\frac{L_{\ep-2}Z_{\ep-2}^{(20)}\si_{\ep-1}^{(11)}\si_{\ep}^{(02)}}{L_{\ep-1}}
\;.
\end{multline}
Further steps are not possible. For the formula to be polynomial
we need that the last two summands vanish. We obtain the equations
\begin{equation}\label{larh}
\la_{\ep-2,\ep-1}:=\si_{\ep-2}^{(20)}\si_{\ep-1}^{(11)}=0,
\qquad
\rh_{\ep-1,\ep}:=\si_{\ep-1}^{(11)}\si_{\ep}^{(02)}=0\:
\end{equation}
in the deformation variables. 

\begin{example}[embedding dimension 5]
The computations up to now suffice. 
We get, modulo the ideal of the base space,
the same equations as equations \eqref{vdvijf}
in Section \ref{inbvijf}. To translate in
the notation used there, note that there are no variables $t_2$ and $t_4$,
so we set $l_2=r_2=l_4=r_4=0$, and we take $l_3=0$, $r_3=-t_3$.
One gets $Z_2^{(k0)}=Z_2^{(a_2-k)}$ and $Z_4^{(0k)}=Z_4^{(a_4-k)}$,
$\si_2^{(20)}=s_2^{(a_2-1)}$ and $\si_4^{(02)}=s_4^{(a_4-1)}$.
For $\ep=3$ we find
\[
Z_3^{(00)}=Z_3^{(01)}R_3=Z_3^{(a_3-1)}(z_3+t_3)=L_3Z_3^{(10)}
=z_3\wt Z_3^{(a_3-1)}
\]
and
\[
Z_3^{(01)}=L_3Z_3^{(11)}+\si_3^{(11)}=z_3Z_3^{(a_3-2)}+s_3^{(a_3-1)}\;.
\]
The formula \eqref{mindrie} gives $\wt Z_3^{(a_3-1)}$ as factor in the
second summand of the right-hand side of the equation $z_1z_5=P_{2,4}$,
but the difference with $ Z_3^{(a_3-1)}$, as given
in the equations \eqref{vdvijf}, lies in the ideal of the base
space. Note that in general 
\begin{multline*}
Z_\ep^{(i+1,j)}-Z_\ep^{(i,j+1)}\\{}=
(Z_\ep^{(i+1,j+1)}R_\ep+\si_\ep^{(i+1,j+1)})-
(L_\ep Z_\ep^{(i+1,j+1)}+\si_\ep^{(i+1,j+1)})
=
(l_\ep-r_\ep)Z_\ep^{(i+1,j+1)}\;.
\end{multline*}
The factor $\si_3^{(11)}$ in the second summand gives that we can use
the equation $(l_3-r_3)\si_3^{(11)}=t_3s_3^{(a_3-1)}=0$.
\end{example}

We obtained the equations (\ref{larh}) as necessary condition
to find a polynomial $P_{\ep-1,\ep+1}$. We observe that they could
be computed before computing  $P_{\ep-1,\ep+1}$, as they are the result
of suitable substitutions in the right hand side of the equations of the 
previous line: $\la_{\ep-1,\ep}=\si_{\ep-1}^{(20)}\si_{\ep}^{(11)}$ 
is gotten by setting $z_{\ep-1}=l_{\ep-1}$ and $z_\ep=l_\ep$ in
$P_{\ep-1,\ep}=Z_{\ep-1}^{(10)}Z_{\ep}^{(01)}$, while $P_{\ep,\ep+1}$
gives $\rh_{\ep,\ep+1}$ by $z_\ep=r_\ep$ and $z_{\ep+1}=r_{\ep+1}$.

To find the versal deformation in general one has to proceed
in the same way for the higher lines of the pyramid. Arndt 
has shown that this works. 
As the proof is only written in his thesis \cite{ar}, we sketch it 
here.

\begin{theorem}\label{arndtthm}
Let $z_{\de-1}z_{\ep+1}=p_{\de,\ep}$, $2\leq\de\leq\ep\leq e-1$,
be the quasi-determinantal equations for a cyclic quotient singularity $X$
of embedding dimension $e$.
There exists a deformation $L_{\de-1}R_{\ep+1}=P_{\de,\ep}$ of
these equations over a base space, whose ideal $J$
has $\dim T^2_X=(e-2)(e-4)$
generators, being $(l_\ep-r_\ep)\si_e^{(11)}$ for $3\leq\ep\leq e-2$,
$\la_{\de,\ep}$ for $2\leq\de<\ep\leq e-2$, and
$\rh_{\de,\ep}$ for $3\leq\de<\ep\leq e-1$. The polynomials $P_{\de,\ep}$
can be determined inductively, followed by 
$\la_{\de,\ep}=P_{\de,\ep}|_{z_\be=l_\be}$ and 
$\rh_{\de,\ep}=P_{\de,\ep}|_{z_\be=r_\be}$, where $\de\leq\be\leq\ep$.
This deformation is versal.
\end{theorem}

\begin{proof}[Sketch of proof]
To find $P_{\de,\ep}$ we have to express the product  $L_{\de-1}R_{\ep+1}$
in the local ring in terms of variables with indices between $\de$ and
$\ep$. We assume that we already have
the equations $L_{\be-1}R_{\ga+1}=P_{\be,\ga}$ for $\ga-\be<\ep-\de$,
and also the base equations formed from them. Let $I_{\de,\ep}$ be the
ideal of all these equation. Obviously $P_{\de,\ep}$ has to 
satisfy
\begin{equation}\label {bgde}
L_\be R_\ga P_{\de,\ep}\equiv P_{\be+1,\ep}P_{\de,\ga-1}
\bmod I_{\de,\ep}
\end{equation}
for all $\be$, $\ga$, and it can be determined from any
of these equations. The other ones then follow. For the
actual computation (following \cite{sb}) we  use $\be=\ep-1$
and $\ga=\ep$, but now we take $\be=\de$, $\ga=\ep$, so the
right hand side of equation (\ref{bgde}) becomes
$P_{\de,\ep-1}P_{\de+1,\ep}$.
We perform successively division with remainder by $L_\be$
and find
\[
P_{\de,\ep-1}=\sum_{\be=\de}^{\ep-1} P_{\de,\ep-1}^{(\be)} L_\be\;,
\]
without remainder because of the equation $\la_{\de,\ep-1}$.
Now we use the congruences
\[
L_\be P_{\de+1,\ep}\equiv L_\de P_{\be+1,\ep}\;,
\]
whose validity one sees upon multiplying with $R_{\ep+1}$.
We conclude that
\[
P_{\de,\ep-1}P_{\de+1,\ep}\equiv
          L_\de (\sum P_{\de,\ep-1}^{(\be)}P_{\be+1,\ep}) \;.
\]
Likewise, from 
\begin{equation}\label{expan}
P_{\de+1,\ep}=\sum Q_{\de+1,\ep}^{(\be)} R_\be\;,
\end{equation}
we get, using $\rh_{\de+1,\ep}$,  that
\[
P_{\de,\ep-1}P_{\de+1,\ep}\equiv
          R_\ep (\sum Q_{\de+1,\ep}^{(\be)}P_{\de,\be-1}) \;.
\]
Arndt proves that, 
if a polynomial is divisible by $L_\de$ and by $R_\ep$, then
it is divisible by the product $L_\de R_\ep$.
To check the statement it suffices to do it for the special fibre
(according to \cite[1.2.2]{ar}). One notes that
the ideal $I_{\de,\ep}$ defines a flat deformation of the product
of a certain cyclic quotient singularity in the variables $z_\de$,
\dots, $z_\ep$ with a smooth factor of the remaining coordinates,
so $z_\de=u^{n'}$, $z_\ep=v^{n'}$
for a certain $n'$. Here indeed it holds, that if 
a polynomial is divisible by $u^{n'}$ and by $v^{n'}$, then
it is divisible by the product $(uv)^{n'}$.
Therefore there exists a polynomial
$P_{\de,\ep}$ with $L_\de R_\ep P_{\de,\ep}= P_{\de,\ep-1}P_{\de+1,\ep}$.

We do not know very much about $ P_{\de,\ep}$. We know that over the
Artin component $Z_{\ep-1}^{(01)}$ is divisible by $L_{\ep-1}$, so
we can define inductively $P_{\de,\ep}|_\text{AC}
= P_{\de,\ep-1}|_\text{AC}Z_\ep^{(01)}/L_{\ep-1}$. Restricted to the
Artin component, the difference between the so defined
$P_{\de,\ep}|_\text{AC}$ and $ P_{\de,\ep}$ from above lies in the 
restriction of the ideal $I_{\de,\ep}$. By induction the elements
of this ideal extend in the correct way, so we can use them to change
$ P_{\de,\ep}$, so that its restriction is equal to 
$P_{\de,\ep}|_\text{AC}$.

To show flatness we lift the relations. On the Artin component
we have the quasi-determinantal relations, which come in two types,
depending on the use of the bottom or top line of the quasi-matrix.
We give the lift for one type, the other being similar. On the 
Artin component one has the relation
\[
L_{\ga-1}(L_{\de-1}R_{\ep+1}-P_{\de,\ep})=
L_{\de-1}(L_{\ga-1}R_{\ep+1}-P_{\ga,\ep})
+
\frac{P_{\ga,\ep}}{R_{\ga}}(L_{\de-1}R_{\ga}-P_{\de,\ga-1})\;,
\]
so using an expansion like (\ref{expan}) for $P_{\ga,\ep}$
we find modulo the ideal
$I_{\de,\ep}$ the relation
\[
L_{\ga-1}(L_{\de-1}R_{\ep+1}-P_{\de,\ep})\equiv 
L_{\de-1}(L_{\ga-1}R_{\ep+1}-P_{\ga,\ep})
+
\sum Q_{\ga,\ep}^{(\be)}(L_{\de-1}R_{\be}-P_{\de,\be-1})\;.
\]

For versality one needs firstly the surjectivity of the map of the
Zariski tangent space of our deformation to $T^1_X$, and secondly the 
injectivity of the obstruction map 
$\Ob\colon(J/\mathfrak m J)^*\to T^2_X$,
where $J$ is the ideal of the base space. That we cover all possible
infinitesimal deformations, is something we have already said and used; 
for a proof (which requires an explicit description of $T^1_X$),
see \cite{ri}, \cite{ar} or \cite{js}. We neither give here an explicit
description of $T^2_X$ (see \cite{ar}). For the map $\Ob$ one starts
with a map $l\colon J/\mathfrak mJ \to \sier X$ 
and exhibits the following
function on relations: consider a relation $\boldsymbol r$, i.e.,
$\sum f_ir_j=0$, which lifts to $\sum F_iR_i=\sum g_j q_j$, where 
the $g_j$ are the generators of the ideal $J$. 
Then $\Ob(l)(\boldsymbol r)=
\sum l(g_j)q_j\in \sier X$. 
From our description of the relations we see that 
the equation $\rh_{\de+1,\ep}$ occurs for the first time
when lifting the relation 
\[
L_{\de}(L_{\de-1}R_{\ep+1}-P_{\de,\ep})=
L_{\de-1}(L_{\de}R_{\ep+1}-P_{\de+1,\ep})
+
\frac{P_{\de+1,\ep}}{R_{\de+1}}(L_{\de-1}R_{\de+1}-P_{\de,\de})\;.
\]
This more or less shows that one really needs all equations for the base
space.
\end{proof} 

Note that this result indeed determines the ideal of the base space,
but does not give explicit formulas for specific generators. Looking
at the equations, say for the total space, one might recognise
the numbers $[1,2,1]$ and $[2,1,2]$, suggesting that an explicit
formula can be somehow given in terms of Catalan combinatorics.
To show that the situation is more complicated, we will derive the
equations of the next line.

We compute the polynomial $P_{\ep-3,\ep}$. It will be more complicated
than formula \eqref{mindrie}, so better notation is desirable to
increase readability.
Following Brohme \cite{sb} we use a position system. In stead of the
complicated symbols $Z_{\ep-\ga}^{(ij)}$ and $\si_{\ep-\ga}^{(ij)}$
we write only the upper index $ij$; the lower index
is not needed, if we write factors
with the same $\ep-\ga$ below each other. To distinguish between 
$Z_{\ep-\ga}^{(ij)}$ and $\si_{\ep-\ga}^{(ij)}$ we write the $ij$ 
representing $Z_{\ep-\ga}^{(ij)}$  in
bold face. 
\def\bb#1#2{\textbf{\large #1#2}}
\def\bi#1#2{\textbf{\textit{\large #1#2}}}
\def\nr#1#2{({\mathrm #1}.#2)}
\begin{example}
The symbol 
\[
\begin{matrix}
\bb 30 &    \bb 20  & \bb 10  &  \bb 03\\
&20&11 \\
&11&12
\end{matrix}
\]
represents the monomial 
$
Z_{\ep-3}^{(30)}\si_{\ep-2}^{(11)}\si_{\ep-2}^{(20)}Z_{\ep-2}^{(20)}
\si_{\ep-1}^{(12)}\si_{\ep-1}^{(11)}Z_{\ep-1}^{(10)}Z_\ep^{(03)}
$.
\end{example}
A factor $L_{\ep-\ga}$ in the denominator will be represented by
$\overline L$, whereas an extra factor  $L_{\ep-\ga}$ in the numerator
will be written in bold face. We start from $P_{\ep-3,\ep-1}Z_\ep^{(01)}
/L_{\ep-1}$, being the sum of two terms. These will be transformed
using the division with remainder \eqref{links} and \eqref{rechts}
and previous equations. One gets some terms, which occur in the final answer,
and some terms, which will be transformed again. Terms that will be 
transformed, are written in italics, and should be considered erased,
when transformed. So $P_{\ep-3,\ep-1}Z_\ep^{(01)}
/L_{\ep-1}$ is, modulo previous equations, equal to the sum of the 
not italicised terms up to a line in italics, plus the terms in that line
(disregarding all text in between). The final result consists of a 
polynomial of 8 terms,  (P.1) -- (P.8), and two types
of terms with $L_{\ep-1}$ in the denominator, called (L.1) -- (L.4)
and (R.1) -- (R.4). Now we start:
\[
\begin{matrix}
\nr I1 &
\begin{matrix}
\bi 10 & \bi 11 & \bi 01 & \bi 01\\
&&\overline L
\end{matrix}&
\qquad+\qquad&\nr I2&
\begin{matrix}
\bi 20 &   \bi 10&  \bi 02 &  \bi 01\\
&\textit{11}& \overline L
\end{matrix}
\\[7mm]
\nr P1 &
\begin{matrix}
\bb 10 & \bb 11 & \bb 11 & \bb 01
\end{matrix}
&+&\nr P2 &
\begin{matrix}
\bb 20 &   \bb 10&  \bb 12 &  \bb 01\\
&11 
\end{matrix}
\\[7mm]
\nr I3&
\begin{matrix}
\bi 10 & \bi 11 &  & \bi 01\\
&&\textit{11}\\
&&\overline L
\end{matrix}&+&\nr I4 &
\begin{matrix}
\bi 20 &   \bi 10&   &  \bi 01\\
&\textit{11}& \textit{12} \\
&&\overline L
\end{matrix}
\end{matrix}
\]
Now take out factors $L_{\ep-2}$ and $R_\ep$, giving
$L_{\ep-1}Z_{\ep-1}^{(10)}$ and two remainders:
\[
\begin{matrix}
\nr P3 &
\begin{matrix}
\bb 10 & \bb 21 & \bb 10 & \bb 02\\
&&11 
\end{matrix}&\qquad+\qquad&\nr P4&
\begin{matrix}
\bb 20 &   \bb 20&  \bb 10 &  \bb 02\\
&11&12 
\end{matrix}
\\[7mm]
\nr I5&
\begin{matrix}
\bi 10 &  &  & \bi 01\\
&\textit{21}&\textit{11}\\
&&\overline L
\end{matrix}&+&\nr I6&
\begin{matrix}
\bi 20 &   &   &  \bi 01\\
&\textit{20}& \textit{12} \\
&\textit{11}&\overline L
\end{matrix}
\\[7mm]
\nr R1 &
\begin{matrix}
&\bb L{} \\
\bb 10 & \bb 21\\
&&11&02\\
&&\overline L
\end{matrix}
&+&\nr R2&
\begin{matrix}
&\bb L{} \\
\bb 20 &   \bb 20\\
&11& 12&02 \\
&&\overline L
\end{matrix}
\end{matrix}
\]
Taking out $L_{\ep-3}$ and $R_\ep$ from (I.5) and (I.6) gives
$Z_{\ep-2}^{(10)}Z_{\ep-1}^{(01)}$ and two remainders:
\[
\begin{matrix}
\nr I7 &
\begin{matrix}
\bi 20 & \bi 10  & \bi 01 & \bi 02\\
&\textit{21}&\textit{11}\\
&&\overline L
\end{matrix}
&\qquad+\qquad
&\nr I8&
\begin{matrix}
\bi 30 &    \bi 10  & \bi 01  &  \bi 02\\
&\textit{20}& \textit{12} \\
&\textit{11}&\overline L
\end{matrix}
\\[7mm]
\nr L1&
\begin{matrix}
 &  &  & \bb 01\\
20&21&11\\
&&\overline L
\end{matrix}&\qquad+\qquad&\nr L2&
\begin{matrix}
 &   &   &  \bb 01\\
30&20&12 \\
&11&\overline L
\end{matrix}
\\[7mm]
\nr R3&
\begin{matrix}
\bb L{}\\
\bb20\\
&21&11&02\\
&&\overline L
\end{matrix}&\qquad+\qquad&\nr R4&
\begin{matrix}
\bb L{}\\
\bb 30 \\
&20 &12&02 \\
&11&\overline L
\end{matrix}
\end{matrix}
\]
Now we follow the steps of the computation of $P_{\ep-2,\ep}$.
\[
\begin{matrix}
\nr P5 &
\begin{matrix}
\bb 20 & \bb 10  & \bb 11 & \bb 02\\
&21&11
\end{matrix}&\qquad+\qquad&\nr P6 &
\begin{matrix}
\bb 30 &    \bb 10  & \bb 11  &  \bb 02\\
&20&12 \\
&11
\end{matrix}
\\[9mm]
\nr I9 &
\begin{matrix}
\bi 20 & \bi 10  &  & \bi 02\\
&\textit{21}&\textit{11}\\
&&\textit{11}\\
&&\overline L
\end{matrix}&\qquad+\qquad&\nr I{10}&
\begin{matrix}
\bi 30 &    \bi 10  &   &  \bi 02\\
&\textit{20}& \textit{11} \\
&\textit{11}&\textit{12}\\
&&\overline L
\end{matrix}
\\[9mm]
\nr P7&
\begin{matrix}
\bb 20 & \bb 20  & \bb 10 & \bb 03\\
&21&11\\
&&11
\end{matrix}&\qquad+\qquad&\nr P8&
\begin{matrix}
\bb 30 &    \bb 20  & \bb 10  &  \bb 03\\
&20&11 \\
&11&12
\end{matrix}
\\[9mm]
\nr L3 &
\begin{matrix}
\bb 20 &  &  & \bb 02\\
&20&11\\
&21&11\\
&&\overline L
\end{matrix}&\qquad+\qquad&\nr L4&
\begin{matrix}
\bb 30 &   &   &  \bb 01\\
&20&11\\
&20&12 \\
&11&\overline L
\end{matrix}
\\[9mm]
\nr R5&
\begin{matrix}
&\bb L{}\\
\bb20&\bb 20\\
&21&11&03\\
&&11\\
&&\overline L
\end{matrix}&\qquad+\qquad&\nr R6&
\begin{matrix}
&\bb L{}\\
\bb 30&\bb 20 \\
&20 &11&03 \\
&11&12\\
&&\overline L
\end{matrix}
\end{matrix}
\]

\begin{remarks}
1.\enspace
It is easy to see that the terms (L.i) vanish modulo the ideal $J$:
the terms (L.1) and (L.2) vanish because of  
the equation $\la_{\ep-3,\ep-1}$,
and the terms (L.3) and (L.4) both vanish by equation $\la_{\ep-2,\ep-1}$.

The terms (R.i) are more difficult. Taken together, (R.2) and (R.5)
vanish by equation $\rh_{\ep-2,\ep}$.  The term (R.1) vanishes by
$\rh_{\ep-1,\ep}$, as does less evidently (R.3). For (R.6) one uses
$\la_{\ep-2,\ep-1}$. The term (R.4) is the most complicated. We 
multiply $\rh_{\ep-2,\ep}$ with $\si_{\ep-2}^{(20)}$ to get
$
\begin{smallmatrix}
20&12&02\\
11
\end{smallmatrix}
+
\begin{smallmatrix}
20&11&03\\
21&11
\end{smallmatrix}
$,
in which the second summand vanishes by equation $\la_{\ep-2,\ep-1}$.
%\end{remark}
\\   
2.\enspace
%\begin{remark}
The term (P.8) lies in the ideal, so one could leave it out.
However, to have a general formula it is better to keep it.
%\end{remark}
\\
3.\enspace
%\begin{remark}
Arndt \cite{ar} gives a slightly different, more symmetric result (without
showing his computation). He has also 8 terms, (P.8) is missing and
(P.4) is replaced by two terms, which together are equivalent
to it,  modulo the base ideal:
\[
\begin{matrix}
\bb 20 & \bb 20 & \bb 01 & \bb 02\\
&11&12 
\end{matrix}\qquad+\qquad
\begin{matrix}
\bb 20 &   \bb 20&  \bb 02 &  \bb 02\\
&11&11 
\end{matrix}
\]
\end{remarks}

In our computation  we work systematically
from the right to the left. Once we take out a factor $L_\ga R_\ep$
and replace it by $P_{\ga+1,\ep-1}$, we  basically repeat an 
earlier computation. Brohme \cite{sb} has given an inductive formula
for the resulting terms (P.i) and the remainder terms (R.i).
The problem lies in showing that the remainder terms lie in the
ideal of the base space. This problem was solved by Martin Hamm \cite{ha}
on the basis of a more direct, combinatorial description of the
occurring terms. Each term is represented by a rooted tree,
which we draw horizontally (Hamm puts as usually the root at the
top). Accordingly we call for length of a tree, what is usually
called its height.

We consider as example the term (P.7). We first draw
the tree such that each vertex directly correspond to a position
in the symbol for (P.7) above, but then we transform it
so that the bottom line is straight. This will be the way
we draw all trees in the sequel.

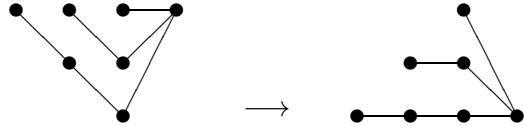
\unitlength 20pt
\[
\begin{picture}(3,2)(0,-2)
\put(0,0)\cir
\put(0,0){\lijn{-1}}
\put(1,0)\cir
\put(1,0){\lijn{-1}}
\put(1,-1)\cir
\put(1,-1){\lijn{-1}}
\put(2,0)\cir
\put(2,0){\lijn{0}}
\put(2,-1)\cir
\put(2,-1){\lijn{1}}
\put(2,-2)\cir
\put(2,-2){\lijn{2}}
\put(3,0)\cir
\end{picture}
\qquad
\lra
\qquad
\begin{picture}(3,2)
\put(0,0)\cir
\put(0,0){\lijn{0}}
\put(1,0)\cir
\put(1,0){\lijn{0}}
\put(1,1)\cir
\put(1,1){\lijn{0}}
\put(2,0)\cir
\put(2,0){\lijn{0}}
\put(2,1)\cir
\put(2,1){\lijn{-1}}
\put(2,2)\cir
\put(2,2){\lijn{-2}}
\put(3,0)\cir
\end{picture}
\]
We  explain how to compute the numbers $ij$ in the symbol from the
tree, with the highest node at distance $\ga$ from the
root giving
$Z_{\ep-\ga}^{(ij)}$, and the other nodes $\si_{\ep-\ga}^{(ij)}$.
Given a tree $T$, the resulting polynomial in these variables will
be denoted by $P(T)$. We write $\la(T)$ for the corresponding
term in $\la_{\de,\ep}$, obtained by putting $z_\be=l_\be$, and
$\rh(T)$ for the term obtained with  $z_\be=r_\be$.

\begin{defn}
Let $T$ be a rooted tree. To each node $a\in T$ we associate two
numbers, $i$ and $j$, as follows. The second number $j$ is the number
of child nodes of $a$. Let $p(a)$ be the parent of $a$; if there exists
a node $b$ lying directly above $a$, let $p(b)$ be its parent.
Then the number $i$ is the number of nodes between $p(a)$ and $p(b)$
(with $p(a)$ and $p(b)$ included), or the number of nodes above
$p(a)$ (with $p(a)$  included), if there is no  node lying above $a$.
\end{defn} 
\noindent
We also represent the remainder terms (R.i) by a tree. 
Such a term has the form $L_\ga R_{\de,\ep}^{(i)}/L_{\ep-1}$. 
The tree will give $R_{\de,\ep}^{(i)}$. If a tree $T$ is given,
we write $R(T)$ for this polynomial.
As example
we consider (R.5). 
We observe that its symbol contains the same pairs of numbers as
(P.7), except that some on the top are missing, and that 
the root represents $\si_\ep^{(ij)}$ instead of $Z_\ep^{(ij)}$.
We represent it by the following tree.
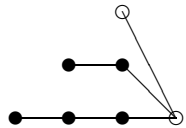
\[
\begin{picture}(3,2)
\put(0,0)\cir
\put(0,0){\lijn{0}}
\put(1,0)\cir
\put(1,0){\lijn{0}}
\put(1,1)\cir
\put(1,1){\lijn{0}}
\put(2,0)\cir
\put(2,0){\lijn{0}}
\put(2,1)\cir
\put(2,1){\lijn{-1}}
\put(2,2)\cio
\put(2,2){\lijn{-2}}
\put(3,0)\cio
\end{picture}
\]
To the open nodes (except the root) we do not attach numbers $ij$, but
these nodes do contribute to the numbers $ij$ for the other nodes.

\begin{example}
For $P_{\ep-3,\ep}$ we find the following trees
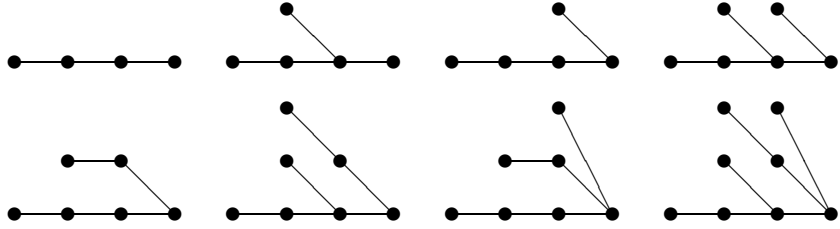
\begin{gather*}
% P.1
\begin{picture}(3,1)
\put(0,0)\cir
\put(0,0){\lijn{0}}
\put(1,0)\cir
\put(1,0){\lijn{0}}
\put(2,0)\cir
\put(2,0){\lijn{0}}
\put(3,0)\cir
\end{picture}
\qquad
% P.2
\begin{picture}(3,1)
\put(0,0)\cir
\put(0,0){\lijn{0}}
\put(1,0)\cir
\put(1,0){\lijn{0}}
\put(1,1)\cir
\put(1,1){\lijn{-1}}
\put(2,0)\cir
\put(2,0){\lijn{0}}
\put(3,0)\cir
\end{picture}
\qquad
% P.3
\begin{picture}(3,1)
\put(0,0)\cir
\put(0,0){\lijn{0}}
\put(1,0)\cir
\put(1,0){\lijn{0}}
\put(2,0)\cir
\put(2,0){\lijn{0}}
\put(2,1)\cir
\put(2,1){\lijn{-1}}
\put(3,0)\cir
\end{picture}
\qquad
% P.4
\begin{picture}(3,1)
\put(0,0)\cir
\put(0,0){\lijn{0}}
\put(1,0)\cir
\put(1,0){\lijn{0}}
\put(1,1)\cir
\put(1,1){\lijn{-1}}
\put(2,0)\cir
\put(2,0){\lijn{0}}
\put(2,1)\cir
\put(2,1){\lijn{-1}}
\put(3,0)\cir
\end{picture}
\\[3mm]
% P.5
\begin{picture}(3,2)
\put(0,0)\cir
\put(0,0){\lijn{0}}
\put(1,0)\cir
\put(1,0){\lijn{0}}
\put(1,1)\cir
\put(1,1){\lijn{0}}
\put(2,0)\cir
\put(2,0){\lijn{0}}
\put(2,1)\cir
\put(2,1){\lijn{-1}}
\put(3,0)\cir
\end{picture}
\qquad
% P.6
\begin{picture}(3,2)
\put(0,0)\cir
\put(0,0){\lijn{0}}
\put(1,0)\cir
\put(1,0){\lijn{0}}
\put(1,1)\cir
\put(1,1){\lijn{-1}}
\put(1,2)\cir
\put(1,2){\lijn{-1}}
\put(2,0)\cir
\put(2,0){\lijn{0}}
\put(2,1)\cir
\put(2,1){\lijn{-1}}
\put(3,0)\cir
\end{picture}
\qquad
% P.7
\begin{picture}(3,2)
\put(0,0)\cir
\put(0,0){\lijn{0}}
\put(1,0)\cir
\put(1,0){\lijn{0}}
\put(1,1)\cir
\put(1,1){\lijn{0}}
\put(2,0)\cir
\put(2,0){\lijn{0}}
\put(2,1)\cir
\put(2,1){\lijn{-1}}
\put(2,2)\cir
\put(2,2){\lijn{-2}}
\put(3,0)\cir
\end{picture}
\qquad
% P.8
\begin{picture}(3,2)
\put(0,0)\cir
\put(0,0){\lijn{0}}
\put(1,0)\cir
\put(1,0){\lijn{0}}
\put(1,1)\cir
\put(1,1){\lijn{-1}}
\put(1,2)\cir
\put(1,2){\lijn{-1}}
\put(2,0)\cir
\put(2,0){\lijn{0}}
\put(2,1)\cir
\put(2,1){\lijn{-1}}
\put(2,2)\cir
\put(2,2){\lijn{-2}}
\put(3,0)\cir
\end{picture}
\end{gather*}
and we have the following trees for the remainder.
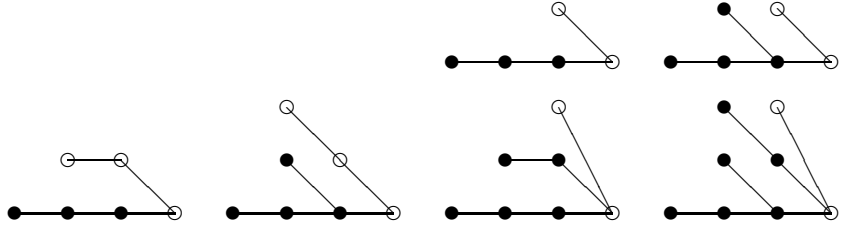
\begin{gather*}
\begin{picture}(3,1)
\end{picture}
\qquad
\begin{picture}(3,1)
\end{picture}
\qquad
% R.1
\begin{picture}(3,1)
\put(0,0)\cir
\put(0,0){\lijn{0}}
\put(1,0)\cir
\put(1,0){\lijn{0}}
\put(2,0)\cir
\put(2,0){\lijn{0}}
\put(2,1)\cio
\put(2,1){\lijn{-1}}
\put(3,0)\cio
\end{picture}
\qquad
% R.2
\begin{picture}(3,1)
\put(0,0)\cir
\put(0,0){\lijn{0}}
\put(1,0)\cir
\put(1,0){\lijn{0}}
\put(1,1)\cir
\put(1,1){\lijn{-1}}
\put(2,0)\cir
\put(2,0){\lijn{0}}
\put(2,1)\cio
\put(2,1){\lijn{-1}}
\put(3,0)\cio
\end{picture}
\\[3mm]
% R.3
\begin{picture}(3,2)
\put(0,0)\cir
\put(0,0){\lijn{0}}
\put(1,0)\cir
\put(1,0){\lijn{0}}
\put(1,1)\cio
\put(1,1){\lijn{0}}
\put(2,0)\cir
\put(2,0){\lijn{0}}
\put(2,1)\cio
\put(2,1){\lijn{-1}}
\put(3,0)\cio
\end{picture}
\qquad
% R.4
\begin{picture}(3,2)
\put(0,0)\cir
\put(0,0){\lijn{0}}
\put(1,0)\cir
\put(1,0){\lijn{0}}
\put(1,1)\cir
\put(1,1){\lijn{-1}}
\put(1,2)\cio
\put(1,2){\lijn{-1}}
\put(2,0)\cir
\put(2,0){\lijn{0}}
\put(2,1)\cio
\put(2,1){\lijn{-1}}
\put(3,0)\cio
\end{picture}
\qquad
% R.5
\begin{picture}(3,2)
\put(0,0)\cir
\put(0,0){\lijn{0}}
\put(1,0)\cir
\put(1,0){\lijn{0}}
\put(1,1)\cir
\put(1,1){\lijn{0}}
\put(2,0)\cir
\put(2,0){\lijn{0}}
\put(2,1)\cir
\put(2,1){\lijn{-1}}
\put(2,2)\cio
\put(2,2){\lijn{-2}}
\put(3,0)\cio
\end{picture}
\qquad
% R.6
\begin{picture}(3,2)
\put(0,0)\cir
\put(0,0){\lijn{0}}
\put(1,0)\cir
\put(1,0){\lijn{0}}
\put(1,1)\cir
\put(1,1){\lijn{-1}}
\put(1,2)\cir
\put(1,2){\lijn{-1}}
\put(2,0)\cir
\put(2,0){\lijn{0}}
\put(2,1)\cir
\put(2,1){\lijn{-1}}
\put(2,2)\cio
\put(2,2){\lijn{-2}}
\put(3,0)\cio
\end{picture}
\end{gather*}
\end{example}

We now characterise the tree representing terms $P(T)$ in the
polynomial $P_{\de,\ep}$, which is obtained as above from
$P_{\de,\ep-1} Z_\ep^{(01)}/L_{\ep-1}$, by working systematically
from the right to the left.
Let $a\in T$ be a node in a rooted tree, 
then $T(a)$ is the (maximal) subtree with 
$a$ as root.

\begin{defn}
An $\al$-tree is a rooted tree satisfying the following property:
if two nodes $a$ and $b$ have the same parent, and if $b$ lies above
$a$, then the subtree $T(b)$ is shorter than $T(a)$.
By $\cal A(k)$ we denote the set of all $\al$-trees of length $k$.
\end{defn}

\begin{theorem}\label{hammthm}
The polynomial $P_{\de,\ep}$ is given by
\[
P_{\de,\ep}=
\sum_{T\in \cal A (\ep-\de+1)} P(T)\;.
\]
\end{theorem}

This Theorem claims two things: that the $\al$-trees give 
exactly the polynomial terms in the computation, and secondly that
this is really the polynomial $P_{\de,\ep}$ we are after. To show the 
second part we have to prove that the remainder terms lie in the
ideal generated by the equations of the base space. 
As we have seen with (R.4) above, the use of
the equations leads to terms, which do not occur in the computation
itself.
We have to characterise
the corresponding trees. This leads to the concept of $\ga$-tree
(Hamm has also $\be$-trees \cite{ha}). We postpone the definition,
and first consider a sub-class of the $\al$-trees.

\begin{defn}
An $\al\ga$-tree is an $\al$-tree, whose root has at least two child nodes
and the subtree of the highest child of the root is unbranched (this
is the chain of open dots in our pictures). Let $\cal{AC}(k,l)$ be
the set of all $\al\ga$-trees of length $k$, such that the unbranched
subtree (with the root included) has length $l$. 
\end{defn}

\begin{lemma}
Modulo the equations $\la_{\ga,\ep-1}$ one has
\[
P_{\de,\ep-1}Z_\ep^{(01)}/L_{\ep-1} =
\sum_{T\in \cal A (\ep-\de)} P(T)
+ \sum_\ga \sum_{T\in \cal {AC} (\ep-\de,\ep-\ga)} L_\ga R(T)/L_{\ep-1} \:.
\]
\end{lemma}

\begin{proof}
We compute as for $P_{\ep-3,\ep}$.
We first consider the rest terms $R(T)$. Such a term comes about from
writing $Z_\ep^{(0,k)}= Z_\ep^{(0,k+1)}R_\ep +\si_\ep^{(0,k+1)}$. We replace
$L_\ga R_\ep$ by $P_{\ga+1,\ep-1}$. The first term of  $P_{\ga+1,\ep-1}$ is
given by an unbranched tree, it is \bb 10 \bb 11 $\cdots$ \bb 11 \bb 01, 
and writing $Z_{\ep-1}^{(01)}= L_{\ep-1}Z_{\ep-1}^{(11)}
+\si_{\ep-1}^{(11)}$ leads to a term $P(T)$ with the same tree as $R(T)$
(with the only difference that all nodes are denoted by black dots).

We are left to show that the polynomial terms are exactly those represented
by $\al$-trees. This is done by induction. 
We can construct an $\al$-tree of length $\ep-\de$ by taking 
a root, an
$\al$-tree of length $\ep-\de-1$ as lowest subtree, and as its complement
an arbitrary $\al$-tree of length at most $\ep-\de-1$ (conversely,
given an  $\al$-tree of length $\ep-\de$, the lowest subtree starting
from the root, but not including it, is $\al$-tree of length $\ep-\de-1$,
while its complement has length at most   $\ep-\de-1$). Doing this in all
possible ways gives all $\al$-trees of length $\ep-\de$.
All these trees occur by our construction: in all monomials of
$P_{\de,\ep-1}Z_\ep^{(01)}/L_{\ep-1}$ we simultaneously take out factors 
$L_\ga$, until we finally are left with 
$\la_{\de,\ep-1}Z_\ep^{(01)}/L_{\ep-1}$. Each  $L_\ga R_\ep$ 
is replaced by $P_{\ga+1,\ep-1}$, and here we repeat the same 
computation as for $P_{\ga+1,\ep}$, except that the upper index of
$Z_\ep^{(0k)}$ is different. This means that we place all possible
trees of length at most $\ep-\ga-1$ above the given tree.
\end{proof}

\begin{remark}
The above proof gives an inductive formula for the number of $\al$-trees
of length $k$:
\[
\# \cal A(k)= \#\cal A(k-1)\cdot \sum_{i=0}^{k-1}\#\cal A(i)\;.
\]
One has $\#\cal A(0)=1$, $\#\cal A(1)=1$, $\#\cal A(2)=2$,
$\#\cal A(3)=8$ and $\#\cal A(4)=96$.
As we have seen in the example, some of the terms lie in the ideal $J$,
generated by the equations of the base space. For $k=4$ already 
$55$ of the $96$ terms lie in $J$, leaving ``only'' 41 terms.
Still this number is considerably larger than the relevant 
Catalan number ($14$ in this case).
\end{remark}

As already mentioned, the use of $\rh$-equations brings us outside 
the realm of $\al$-trees. We retain some properties, which are automatically
satisfied for $\al$-trees. The definition becomes rather involved.

\begin{defn}
A $\ga$-tree of length $k$ is a rooted tree 
satisfying the following properties:\\
(i)
there is only one node at distance $k$ from the root, and it lies on the 
bottom line,\\
(ii) the number of child nodes of a node at distance $d$ from the
root is at most $k-d$,\\
(iii) a node $a$ has a child node, if there exists a node $b$ lying above
$a$ with the same parent,\\
(iv) the root has at least two child nodes
and the subtree of the highest child of the root is unbranched.\\
By $\cal G(k,l)$ we denote the set of all $\ga$-trees of length $k$,
such that the unbranched
subtree (with the root included) has length $l$.
\end{defn}

\begin{example}
We consider the term (R.4) above. The sum of the following two terms
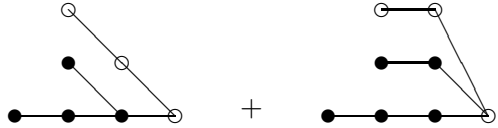
\[
% R.4
\begin{picture}(3,2)
\put(0,0)\cir
\put(0,0){\lijn{0}}
\put(1,0)\cir
\put(1,0){\lijn{0}}
\put(1,1)\cir
\put(1,1){\lijn{-1}}
\put(1,2)\cio
\put(1,2){\lijn{-1}}
\put(2,0)\cir
\put(2,0){\lijn{0}}
\put(2,1)\cio
\put(2,1){\lijn{-1}}
\put(3,0)\cio
\end{picture}
\qquad + \qquad
%% R.4 gamma
\begin{picture}(3,2)
\put(0,0)\cir
\put(0,0){\lijn{0}}
\put(1,0)\cir
\put(1,0){\lijn{0}}
\put(1,1)\cir
\put(1,1){\lijn{0}}
\put(1,2)\cio
\put(1,2){\lijn{0}}
\put(2,0)\cir
\put(2,0){\lijn{0}}
\put(2,1)\cir
\put(2,1){\lijn{-1}}
\put(2,2)\cio
\put(2,2){\lijn{-2}}
\put(3,0)\cio
\end{picture}
\]
is a multiple of the equation $\rh_{\ep-2,\ep}$, and the second graph is
not an $\al$-tree.
\end{example}

We have to show that the sum of all remainder terms (i.e., the
sum of the $R(T)$ over all $\al\ga$-trees) lies in the ideal $J$
generated by the base equations. We do this by showing that the sum of
$R(T)$ over all $\ga$-trees lies in the ideal, as does the sum over all
$\ga$-trees which are not $\al$-trees.

\begin{lemma}\label{laeq}
The sum
$\sum_{T\in \cal{G}(\ep-\de,\ep-\ga) \setminus 
\cal{AG}((\ep-\de,\ep-\ga)} R(T)$
lies in the ideal generated by the $\la$-equations.
\end{lemma}

\begin{proof}
Let $T$ be a $\ga$-tree, which is not an $\al$-tree. Then there exists
a node $a$, such that the subtree $T(a)$ is an $\al$-tree, but directly
above $a$ lies a node $b$ with the same parent, such that the
bottom line of the subtree
$T(b)$ is at least as long as $T(a)$. Denote by $R(T|T(a))$ the monomial
obtained by only multiplying the factors of $R(T)$ corresponding to the
nodes lying in $T(a)$. We claim that $R(T|T(a))=\la(T(a))$. We have to
compute the numbers $ij$. The second number, of child nodes, is 
determined by $T(a)$ only. The number $i$ also coincides in $R(T)$
and $P(T(a))$, except when $c$ is a node without nodes above it in $T(a)$.
Then its value in $R(T)$ is one more than in $P(T(a))$, so the same as
in $\la(P(a))$, proving the claim.
Replacing $T(a)$ in $T$ by an another $\al$-tree of the same length gives a
another $\ga$-tree, which is not an $\al$-tree. So the sum of $R(T)$ over
all $\ga$-trees, differing only in the $\al$-tree with root $a$, is
a multiple of a $\la$-equation. If $T$ has several such subtrees, we 
consider all possible replacements, and get the product of $\la$-equations.  
\end{proof}

The next task is to find terms of $\rh$-equations in a given tree. For
this we introduce the operation of taking away one child node at each
highest node. This can be done for any $\ga$-tree. 

\begin{defn}
Let $T$ be a $\ga$-tree. We determine inductively 
a subtree $G(T)$ with the same root
as $T$ by the following condition: if $a_1$, \dots, $a_p$ are the nodes
in $G(T)$ at distance $d$ from the root,  then they have the same child nodes
in $G(T)$ as in $T$, except for the highest node $a_p$, where we take 
away the highest child node.
\end{defn}

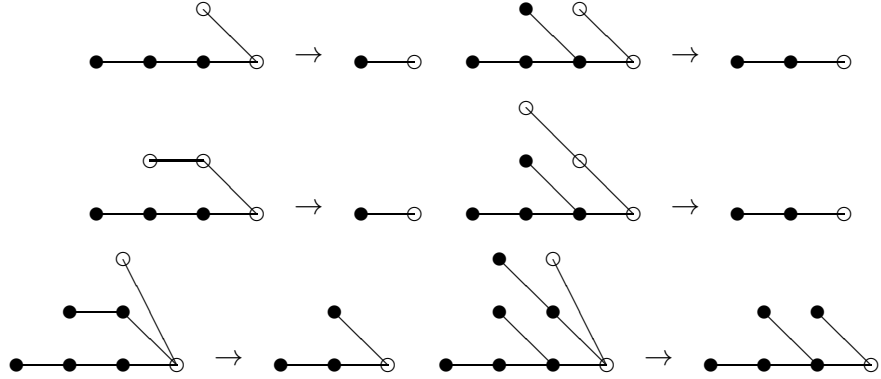
\begin{example}
\begin{gather*}
%  R.1
\begin{picture}(3,1)
\put(0,0)\cir
\put(0,0){\lijn{0}}
\put(1,0)\cir
\put(1,0){\lijn{0}}
\put(2,0)\cir
\put(2,0){\lijn{0}}
\put(2,1)\cio
\put(2,1){\lijn{-1}}
\put(3,0)\cio
\end{picture}
\quad \to \quad
\begin{picture}(1,1)
\put(0,0)\cir
\put(0,0){\lijn{0}}
\put(1,0)\cio
\end{picture}
\qquad
% R.2
\begin{picture}(3,1)
\put(0,0)\cir
\put(0,0){\lijn{0}}
\put(1,0)\cir
\put(1,0){\lijn{0}}
\put(1,1)\cir
\put(1,1){\lijn{-1}}
\put(2,0)\cir
\put(2,0){\lijn{0}}
\put(2,1)\cio
\put(2,1){\lijn{-1}}
\put(3,0)\cio
\end{picture}
\quad \to \quad
\begin{picture}(1,1)
\put(0,0)\cir
\put(0,0){\lijn{0}}
\put(1,0)\cir
\put(2,0)\cio
\put(1,0){\lijn{0}}
\end{picture}
\\[3mm]
% R.3
\begin{picture}(3,2)
\put(0,0)\cir
\put(0,0){\lijn{0}}
\put(1,0)\cir
\put(1,0){\lijn{0}}
\put(1,1)\cio
\put(1,1){\lijn{0}}
\put(2,0)\cir
\put(2,0){\lijn{0}}
\put(2,1)\cio
\put(2,1){\lijn{-1}}
\put(3,0)\cio
\end{picture}
\quad \to \quad
\begin{picture}(1,2)
\put(0,0)\cir
\put(0,0){\lijn{0}}
\put(1,0)\cio
\end{picture}
\qquad
% R.4
\begin{picture}(3,2)
\put(0,0)\cir
\put(0,0){\lijn{0}}
\put(1,0)\cir
\put(1,0){\lijn{0}}
\put(1,1)\cir
\put(1,1){\lijn{-1}}
\put(1,2)\cio
\put(1,2){\lijn{-1}}
\put(2,0)\cir
\put(2,0){\lijn{0}}
\put(2,1)\cio
\put(2,1){\lijn{-1}}
\put(3,0)\cio
\end{picture}
\quad \to \quad
\begin{picture}(1,1)
\put(0,0)\cir
\put(0,0){\lijn{0}}
\put(1,0)\cir
\put(2,0)\cio
\put(1,0){\lijn{0}}
\end{picture}
\\[3mm]
% R.5
\begin{picture}(3,2)
\put(0,0)\cir
\put(0,0){\lijn{0}}
\put(1,0)\cir
\put(1,0){\lijn{0}}
\put(1,1)\cir
\put(1,1){\lijn{0}}
\put(2,0)\cir
\put(2,0){\lijn{0}}
\put(2,1)\cir
\put(2,1){\lijn{-1}}
\put(2,2)\cio
\put(2,2){\lijn{-2}}
\put(3,0)\cio
\end{picture}
\quad \to \quad
\begin{picture}(2,2)(1,0)
\put(1,0)\cir
\put(1,0){\lijn{0}}
\put(2,0)\cir
\put(2,0){\lijn{0}}
\put(2,1)\cir
\put(2,1){\lijn{-1}}
\put(3,0)\cio
\end{picture}
\qquad
% R.6
\begin{picture}(3,2)
\put(0,0)\cir
\put(0,0){\lijn{0}}
\put(1,0)\cir
\put(1,0){\lijn{0}}
\put(1,1)\cir
\put(1,1){\lijn{-1}}
\put(1,2)\cir
\put(1,2){\lijn{-1}}
\put(2,0)\cir
\put(2,0){\lijn{0}}
\put(2,1)\cir
\put(2,1){\lijn{-1}}
\put(2,2)\cio
\put(2,2){\lijn{-2}}
\put(3,0)\cio
\end{picture}
\quad \to \quad
\begin{picture}(3,2)
\put(0,0)\cir
\put(0,0){\lijn{0}}
\put(1,0)\cir
\put(1,0){\lijn{0}}
\put(1,1)\cir
\put(1,1){\lijn{-1}}
\put(2,0)\cir
\put(2,0){\lijn{0}}
\put(2,1)\cir
\put(2,1){\lijn{-1}}
\put(3,0)\cio
\end{picture}
\end{gather*}
\end{example}

\begin{lemma}
Suppose $G(T)$ is an $\al$-tree. Then $R(T|G(T))=\rh(G(T))$ if and only
if the number of child nodes in $T$ is at least 1 for every highest
lying node in $G(T)$.
\end{lemma}

The sum $\sum R(T)$ over all trees satisfying the conditions of the lemma
lies in the ideal generated by the $\rh$-equations.
If the number of child nodes in $T$ is at least 1 for every highest
lying node in $G(T)$, but $G(T)$ is not an $\al$-tree, then one can
find as before a term of a $\la$-equation.

The most difficult case is when the condition on the number of child nodes
is not satisfied. An example is the remainder term (R.6), which is 
represented by the last two pictures in the previous example.
The term contains a factor, which is a term in a $\la$-equation, but
the corresponding dots are not connected by an edge. There is a way
to connect the edges differently, bringing a $\la$-term into evidence.
For this we refer to \cite[pp.~30--40]{ha}. We conclude:

\begin{lemma}
The sum
$\sum_{T\in \cal{G}(\ep-\de,\ep-\ga)} R(T)$
lies in the ideal generated by the $\la$ and $\rh$-equations.
\end{lemma}

Together with Lemma \ref{laeq} this shows that the remainder 
$\sum_{T\in \cal{AG}(\ep-\de,\ep-\ga)} R(T)$ lies in the ideal $J$, thereby
concluding the proof of Theorem \ref{hammthm}.

\def\sig#1#2#3{\si_{#1}^{(#2#3)}}
\begin{example}[The base space for $e=6$, see \cite{ar,sb}]
There are 8 equations, which read as
\begin{gather*}
(l_3-r_3)\sig311, \qquad (l_4-r_4)\sig411, \\
\sig220\sig311, \qquad \sig320\sig411, \qquad \sig311\sig402,
\qquad\sig411\sig502,\\
\sig220\sig321\sig411+\sig230\sig320\sig311\sig412, 
\qquad
\sig311\sig412\sig502+\sig321\left(\sig411\right)^2\sig503.
\end{gather*}
We note the relations 
\[
\sig320-\sig311=(l_3-r_3)\sig321,
\qquad
\sig402-\sig411=(r_4-l_4)\sig412.
\]
We can take $\sig220$, $\sig230$, $l_3-r_3$, $\sig311$, $\sig321$,
$r_4-l_4$, $\sig411$, $\sig412$, $\sig502$ and $\sig503$ as independent
coordinates.
The relation with the coordinates in Section \ref{cqs}, see formula
\eqref{infdef}, is the following:
$l_3-r_3=t_3$, $\sig311=s_3^{(a_3-1)}$,  $\sig312=s_3^{(a_3-2)}$.
Also for $\ep=2$ and $\ep=5$ it is simple:
$\sig\ep i j=s_\ep^{(a_\ep-i-j)}$, but for $\ep=4$ it is more complicated:
$r_4-l_4=-t_4$, $\sig411=s_4^{(a_4-1)}$, while 
$\sig412=\tilde s_4^{(a_4-2)}$, where $\tilde s_\ep^{(\nu)}$
is defined as following \cite[5.1.1]{ar}, see also
\cite[p.~38]{sb}, by the equality
\[
Z_\ep^{(00)}=
(z_\ep+t_\ep)\sum_{\nu=0}^{a_\ep-1}s_\ep^{(\nu)}z_\ep^{a_\ep-1-\nu}=
z_\ep\sum_{\nu=0}^{a_\ep-1}\tilde s_\ep^{(\nu)}(z_\ep+t_\ep)^{a_\ep-1-\nu},
\]
where we put $s_\ep^{(0)}=1$. This implies that
\[
\tilde s_\ep^{(\nu)}=\sum_{\mu=0}^\nu \binom{a_\ep-2-\mu}{a_\ep-2-\nu}
(-t_\ep)^{\mu-\nu}s_\ep^{(a_\ep-1-\mu)}\;.
\]
The primary decomposition gives  five reduced components and one 
embedded component. The five  components are parametrised by the five
sparse coloured triangles of height 2.
\unitlength=1pt
\[
\begin{matrix}
\begin{picture}(50,30)
\put(0,0){\wit}
\put(30,0){\wit}
\put(15,15){\wit}
\end{picture}\colon&
\sig311=\sig411=0\hfill\\[1mm]
\begin{picture}(50,30)
\put(0,0){\zwart}
\put(30,0){\wit}
\put(15,15){\wit}
\end{picture}\colon&
\sig220=l_3-r_3=\sig411=\sig412=0\hfill\\[1mm]
\begin{picture}(50,30)
\put(0,0){\wit}
\put(30,0){\zwart}
\put(15,15){\wit}
\end{picture}\colon&
\sig311=\sig321=r_4-l_4=\sig502=0\hfill\\[1mm]
\begin{picture}(50,30)
\put(0,0){\zwart}
\put(30,0){\wit}
\put(15,15){\zwart}
\end{picture}\colon&
\sig220=\sig230=l_3-r_3=r_4-l_4=\sig411=\sig502=0\\[1mm]
\begin{picture}(50,30)
\put(0,0){\wit}
\put(30,0){\zwart}
\put(15,15){\zwart}
\end{picture}\colon&
\sig220=l_3-r_3=\sig311=r_4-l_4=\sig502=\sig503=0
\end{matrix}
\]
The embedded component is supported at
$\sig220=l_3-r_3=\sig311=r_4-l_4=\sig411=\sig502=0$, which is the locus 
of singularities of embedding dimension 6.
\end{example}

%Specialisation
The primary decomposition, given in the example above, holds
if all $a_\ep$ are large enough, meaning that $a_\ep\geq \max(k_\ep)$,
where the $k_\ep$ depend on the possible triangles.
By openness of versality one deduces the structure for all cyclic
quotient singularities of the given embedding dimension.
The formulas for the base space and the total space of the deformation
hold in all cases, with suitable interpretations.
We note that $Z_\ep^{(ij)}$ is a monic polynomial in $z_\ep$
of degree $a_\ep-i-j$, obtained by division with remainder.
Therefore $Z_\ep^{(ij)}=1$ if $i+j=a_\ep$ and 
$Z_\ep^{(ij)}=0$ if $i+j>a_\ep$. For the remainder terms $\si_\ep^{(ij)}$
we find therefore that $\si_\ep^{(ij)}=1$ if $i+j=a_\ep+1$ and 
$\si_\ep^{(ij)}=0$ if $i+j>a_\ep+1$. By using these values the formulas
hold.
Also the description of the reduced components holds in general, if
one takes an equation $1=0$ to mean that the component is absent.

\begin{example}[The cone over the rational normal curve \cite{ar,sb}]
\begin{proposition}
For the cone over the rational normal curve of degree $e-1$ the
versal deformation is given by the equations $L_{\de-1} R_{\ep+1} = 
P_{\de,\ep}$ with $P_{\ep,\ep}=Z_\ep^{(00)}$ and
\[
P_{\de,\ep} =
 Z_\de^{(10)}Z_\ep^{(01)}+
  \sum_{\ga=1}^{\ep-\de-1} \si_{\ep-\ga}^{(11)}Z_{\de+\ga}^{(10)}
\]
for $\ep-\de>0$.
\end{proposition}

\begin{proof}
We derive the formula from Hamm's description of the equations
(i.e., Theorem \ref{hammthm}).
All terms in the equations containing $Z_\ep^{(ij)}$ with $i+j>2$
and $\si_\ep^{(ij)}$ with $i+j>3$ are absent.
We characterise the remaining $\al$-trees. We have of course a
simple chain, giving rise to the first term in the formula.
Suppose we have a factor $\si_{\ep-\ga}^{(11)}$, coming from 
a node $a$  on the bottom line. Then there is a node lying directly above
it, having the same parent. 
The unique child node of $a$ has $i=2$. If it has itself a child node,
then necessarily $ij=21$ , and there is a node lying above it.
This process continues until we come to an end node with $ij=20$.
If the parent of $a$  is not the root, 
then necessarily  $ij=12$ for it, so there is  a node
lying above it. For the node above $a$ we find then that $i=2$,
and there lies a whole chain above the chain starting with this node.
In this way we proceed to the root, which has $ij=02$.
We find that the tree has the following shape: from each node on 
the right of the node $a$ originates a chain
of maximal length. An example is
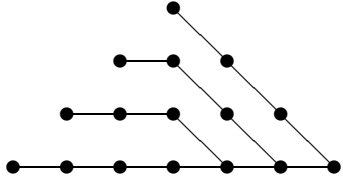
\[
\unitlength=20pt
\begin{picture}(6,3)
\put(0,0)\cir
\put(0,0){\lijn{0}}
\put(1,0)\cir
\put(1,0){\lijn{0}}
\put(1,1)\cir
\put(1,1){\lijn{0}}
\put(2,0)\cir
\put(2,0){\lijn{0}}
\put(2,1)\cir
\put(2,1){\lijn{0}}
\put(2,2)\cir
\put(2,2){\lijn{0}}
\put(3,0)\cir
\put(3,0){\lijn{0}}
\put(3,1)\cir
\put(3,1){\lijn{-1}}
\put(3,2)\cir
\put(3,2){\lijn{-1}}
\put(3,3)\cir
\put(3,3){\lijn{-1}}
\put(4,0)\cir
\put(4,0){\lijn{0}}
\put(4,1)\cir
\put(4,1){\lijn{-1}}
\put(4,2)\cir
\put(4,2){\lijn{-1}}
\put(5,0)\cir
\put(5,0){\lijn{0}}
\put(5,1)\cir
\put(5,1){\lijn{-1}}
\put(6,0)\cir
\end{picture}
\]
Finally we observe that the lowest lying child node of the root
or of a node with $ij=12$ cannot have $i=2$, and that the next to last
node on the left cannot have $ij=12$, so there has  to be a node
$a$ with $ij=11$.

Consider now $P(T)$, if $T$ is not a chain.
The only node with $ij=10$ lies as end-node on the highest chain,
so indeed $P(T)=\si_{\ep-\ga}^{(11)}Z_{\de+\ga}^{(10)}$.
\end{proof}

It follows that 
\[
\la_{\de,\ep} =
  \sum_{\ga=0}^{\ep-\de-1} \si_{\ep-\ga}^{(11)}\si_{\de+\ga}^{(20)}
\]
and
\[
\rh_{\de,\ep} =
 \si_\de^{(11)}\si_\ep^{(02)}+
  \sum_{\ga=1}^{\ep-\de-1} \si_{\ep-\ga}^{(11)}\si_{\de+\ga}^{(11)}\;.
\]
With $l_\ep-r_\ep=t_\ep$, $\si_\ep^{(11)}=s_\ep$, $\si_\ep^{(20)}=s_\ep+t_\ep$
and $\si_\ep^{(02)}=s_\ep-t_\ep$ we get the same formulas as Arndt
gives \cite[5.1.4]{ar}.

Note that $(\si_\ep^{(11)})^3=\si_\ep^{(11)}\rh_{\ep-1,\ep+1}-
\si_{\ep-1}^{(11)}\rh_{\ep-1,\ep}$, so $\si_\ep^{(11)}$ lies in the
radical of the ideal for $3<\ep<e-1$; for $2<\ep<e-2$ one has a formula
with $\la$-equations. So indeed the Artin component is the only
component, if $e>5$.
\end{example}

Other applications of the explicit equations include
\begin{itemize}
\item the discriminant of the components and adjacencies, studied
by Christophersen \cite{ch} and Brohme \cite{sb},
\item embedded components. For low embedding dimension Brohme found 
all components. He made a general conjecture
\cite[4.4]{sb}.
\end{itemize}

\section{Reduced base space}
The ideal $J$ of the base space is described explicitly by Theorems
\ref{arndtthm} and \ref{hammthm}. We have to determine the radical
$\sqrt J$ of this ideal. We are able to do this explicitly for
low embedding dimension, and formulate a conjecture in general. We prove
that the proposed ideal describes the reduced components. The combinatorics
involved resembles that described by Jan Christophersen in his thesis
\cite{ch2}. To prove the conjectural part one has to show that the
monomials we give below, really lie in $\sqrt J$, something we do not do here.

\begin{example}[$e=6$ continued] 
We multiply $\rh_{3,5}$, the last one  of the 8 equations
for the base space, by $\sig411$.
Then the first summand contains the factors $\sig411\sig502$ so lies in the
ideal $J$. Therefore also the second term 
$\sig321\left(\sig411\right)^3\sig503$
lies in $J$, and  $\sig321\sig411\sig503$ lies in the radical $\sqrt J$. Then 
also the first summand of $\rh_{3,5}$ lies in $\sqrt J$. If we
multiply $\la_{2,4}$ with $\sig411$, then the second summand lies in $J$.
We find that the first summand of $\la_{2,4}$ lies in the radical, so also
the second summand. One has $\sig311(\sig320-\sig311)=\sig311(l_3-r_3)
\sig321$, which lies in the ideal, so not only the second summand
$\sig230\sig320\sig311\sig412$, 
but also $\sig230\left(\sig311\right)^2\sig412$ and therefore
$\sig230\sig311\sig412$ lie in $\sqrt J$.
We find the following equations
\begin{gather*}
(l_3-r_3)\sig311, \qquad (l_4-r_4)\sig411, \\
\sig220\sig311, \qquad \sig320\sig411, \qquad \sig311\sig402,
\qquad\sig411\sig502,\\
\sig220\sig321\sig411, \qquad \sig230\sig311\sig412, 
\qquad
\sig311\sig412\sig502, \qquad \sig321\sig411\sig503.
\end{gather*}
This ideal is not reduced, as it contains $\left(\sig311\right)^2\sig411=
\sig311\sig320\sig411-\sig311(l_3-r_3)\sig321\sig411$, but not
$\sig311\sig411$. But it is easy to find the reduced components from the
given equations.
\end{example}

Our first, rough conjecture is that each summand of the equations
$\la_{\de,\ep}$, $\rh_{\de,\ep}$ lies in the radical $\sqrt J$. 
Let us look at
$\rh_{\ep-3,\ep}$. We note that (P.4) and (P.5) yield the same term,
being 
$
\begin{smallmatrix}
21 & 21 & 11 & 03 \\
  & 11  &  12
\end{smallmatrix}
$
and 
$
\begin{smallmatrix}
21 & 11 & 12 & 03 \\
  & 21  &  11
\end{smallmatrix}
$
respectively. As we have the equation 
$
\begin{smallmatrix}
21 & 11 & 03 
\end{smallmatrix}
$
in the radical, these terms do not contribute new equations. As (P.8) itself
already lies in the ideal, we are left with 5 terms (a Catalan number!).
One computes that indeed each summand lies in the radical. We look at
the term in $\rh_{\ep-3,\ep}$, coming from (P.6):
\[
\begin{matrix}
31 & 11 & 12 & 03 \\
  & 20  &  12 \\
& 11
\end{matrix}\;.
\]
\unitlength=1pt
We claim that it is associated to the extended triangle
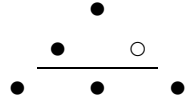
\[
\begin{picture}(75,36)(-3,-3)
\put(0,0){\zwart}
\put(30,0){\zwart}
\put(60,0){\zwart}
\put(7.5,7.5){\line(1,0){45}}
\put(15,15){\zwart}
\put(45,15){\wit}
\put(30,30){\zwart}
\end{picture}
\]
The easiest way to see this is via the numbers $k_\ep$ and $\al_\ep$, 
being $\bcf k = [3,1,2,2]$ and $ (\boldsymbol \al)=(1,3,2,1)$ in this case.
One sees that there are $\al_\ep$ factors $\sig\ep ij$, and they all have
$i+j=k_\ep+1$. The other terms can be parametrised in the same way by the
other extended triangles. The same picture parametrises the term in
$\la_{\ep-3,\ep}$, coming from (P.6): 
\[
\begin{matrix}
40 & 20 & 21 & 12 \\
  & 20  &  12 \\
& 11
\end{matrix}\;.
\]
In the radical we find 
$
\begin{smallmatrix}
31 & 11 & 12& 03 
\end{smallmatrix}
$
and 
$
\begin{smallmatrix}
40 & 11 & 12&12 
\end{smallmatrix}
$.
For the last term we compute as follows:
$\sig{\ep-1}12(\sig{\ep-1}21-\sig{\ep-1}12)=
\sig{\ep-1}12(l_{\ep-1}-r_{\ep-1})\sig{\ep-1}22=
(\sig{\ep-1}11-\sig{\ep-1}02)\sig{\ep-1}22$, and we observe that
the term contains the factors
$\sig{\ep-2}20$ and $\sig{\ep-2}11$.

We can now make our conjecture more precise. As remarked before,
we  do not quite get the radical $\sqrt J$ of the ideal of the base space,
but an intermediate ideal, obtained from the summands
in the generators of $J$. As variables we use $l_\ep-r_\ep$,
and the $\sig\ep ij$, which are connected by the relations
\[
\sig\ep{i+1,}j-\sig\ep i{,j+1}=(l_\ep-r_\ep)\sig\ep{i+1,}{j+1}\:.
\]
\begin{conjecture}
For the ideal $J$ of the base space of the versal deformation of 
a cyclic quotient singularity of embedding dimension $e$ and
its radical $\sqrt J$ holds that
$\sqrt J=\sqrt {J'}\supset J' \supset J$, where $J'$ is the
ideal generated by $(l_\ep-r_\ep)\sig\ep11$,
for $2<\ep<e-1$ and monomials $\la(\triangle_{\de,\ep})$, 
$2\leq \de<\ep<e-1$, and $\rh(\triangle_{\de,\ep})$, 
$2<\de<\ep\leq e-1$, parametrised by sparse coloured triangles 
$\triangle_{\de,\ep}$,
of the form
$
\prod_{\be=\de}^\ep \sig\be{i_\be}{j_\be}\;.
$
The numbers $i_\be$, $j_\be$ are determined as follows:
if $\al_\be>1$, then in both $\la(\triangle_{\de,\ep})$
and $\rh(\triangle_{\de,\ep})$
\begin{align*}
i_\be&= \#\{\text{black dots on right half-line $l_\ep$}\} \\
j_\be&= \#\{\text{black dots on left half-line $l_\ep$}\} 
\end{align*}
but if $\al_\be=1$, then in $\la(\triangle_{\de,\ep})$
\begin{align*}
i_\be&= \#\{\text{black dots on right half-line $l_\ep$}\}+1 \\
j_\be&= \#\{\text{black dots on left half-line $l_\ep$}\} 
\end{align*}
and in $\rh(\triangle_{\de,\ep})$
\begin{align*}
i_\be&= \#\{\text{black dots on right half-line $l_\ep$}\} \\
j_\be&= \#\{\text{black dots on left half-line $l_\ep$}\}+1 
\end{align*}
\end{conjecture}

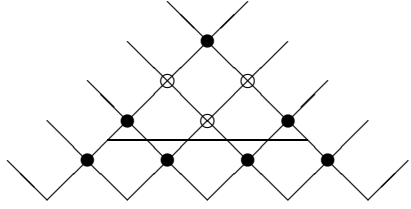
\begin{example}
\[
\begin{picture}(150,80)(-30,-15)
\put(0,0){\zwart}
\put(30,0){\zwart}
\put(60,0){\zwart}
\put(90,0){\zwart}
\put(-15,-15){\line(1,1){75}}
\put(-15,-15){\line(-1,1){15}}
\put(15,-15){\line(1,1){60}}
\put(15,-15){\line(-1,1){30}}
\put(45,-15){\line(1,1){45}}
\put(45,-15){\line(-1,1){45}}
\put(75,-15){\line(1,1){30}}
\put(75,-15){\line(-1,1){60}}
\put(105,-15){\line(1,1){15}}
\put(105,-15){\line(-1,1){75}}
\put(7.5,7.5){\line(1,0){75}}
\put(15,15){\zwart}
\put(45,15){\wit}
\put(75,15){\zwart}
\put(30,30){\wit}
\put(60,30){\wit}
\put(45,45){\zwart}
\end{picture}
\]
One has $\la(\triangle_{\de,\ep})=
\begin{smallmatrix}
40&11&22&11&13
\end{smallmatrix}
$
and
$\rh(\triangle_{\de,\ep})=
\begin{smallmatrix}
31&11&22&11&04
\end{smallmatrix}$.
\end{example}

\begin{remark} The generators of $J'$ correspond to certain terms in generators
of $J$, so there is a special
subclass of $\al$-trees, counted by the Catalan numbers. 
It would be interesting to characterise them. The five 
trees of length 3 can be seen from the previous pictures. We now
list all 14 trees of length 4.
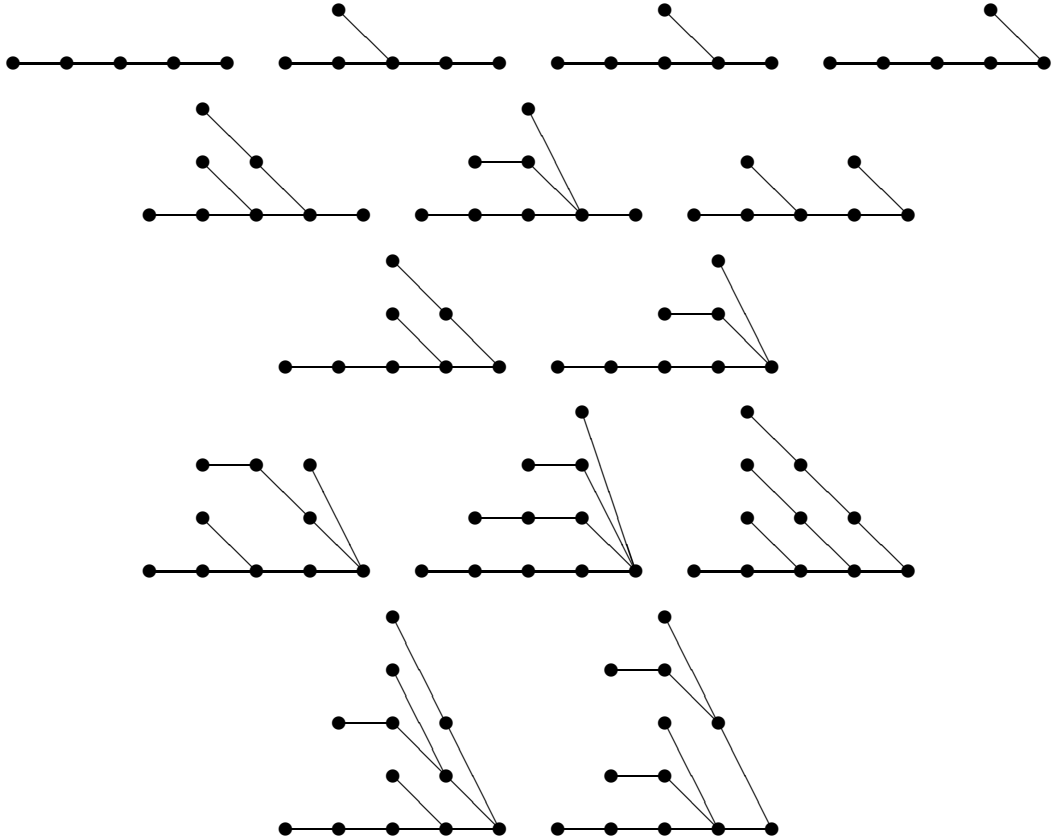
\unitlength=20pt
\def\pcl#1#2#3{\put(#1,#2)\cir \put(#1,#2){\lijn{-#3}}}
\def\base{\pcl000 \pcl100 \pcl 200 \pcl300 \put(4,0)\cir}
\begin{gather*}
\begin{picture}(4,1)
\base
\end{picture}
\qquad
\begin{picture}(4,1)
\base \pcl111
\end{picture}
\qquad
\begin{picture}(4,1)
\base \pcl211
\end{picture}
\qquad
\begin{picture}(4,1)
\base \pcl311
\end{picture}
\\[3mm]
\begin{picture}(4,2)
\base \pcl111 \pcl121 \pcl211
\end{picture}
\qquad
\begin{picture}(4,2)
\base \pcl110 \pcl211 \pcl222
\end{picture}
\qquad
\begin{picture}(4,2)
\base \pcl111 \pcl311 
\end{picture}
\\[3mm]
\begin{picture}(4,2)
\base \pcl211 \pcl 221 \pcl311 
\end{picture}
\qquad
\begin{picture}(4,2)
\base \pcl210 \pcl311 \pcl322 
\end{picture}
\\[3mm]
\begin{picture}(4,3)
\base \pcl111 \pcl120 \pcl221 \pcl 311 \pcl322
\end{picture}
\qquad
\begin{picture}(4,3)
\base \pcl110 \pcl210 \pcl220 \pcl311 \pcl322 \pcl333 
\end{picture}
\qquad
\begin{picture}(4,3)
\base \pcl111 \pcl121 \pcl131 \pcl211 \pcl221 \pcl311 
\end{picture}
\\[3mm]
\begin{picture}(4,4)
\base \pcl120 \pcl211 \pcl221 \pcl232 \pcl242 \pcl311 \pcl322
\end{picture}
\qquad
\begin{picture}(4,4)
\base \pcl110 \pcl130 \pcl211 \pcl 222 \pcl231 \pcl242 \pcl322 
\end{picture}
\end{gather*}
\end{remark}

Inductive proofs about the reduced components often use the
procedure of blowing up and blowing down \cite[1.1]{js}. The
term comes from the analogy with chains of rational curves on a smooth
surface, which can be described by continued fractions.
For sparse coloured triangles it means the following \cite[Lemma 1.8]{sb}.

Blowing up is a way to obtain an extended triangle of height $e-2$
from an extended triangle of height $e-3$. 
Choose  an index $2\leq \ep\leq e$. We define a shift function
$s\colon \{2,\dots,e-1\} \to  \{2,\dots,\ep-1\} \cup
 \{\ep+1,\dots,e\} $ by $s(\be)=\be$ if $\be\in \{2,\dots,\ep-1\} $
and $s(\be)=\be+1$ if $\be\in \{\ep,\dots,e-1\}$.
The blow-up  $\Bl_\ep({\underline\triangle})$  of $\underline\triangle$ at
the index $\ep$ is the triangle with $(s(\be),s(\ga))\in
B(\Bl_\ep({\underline\triangle}))$ if and only if $(\be,\ga)\in
B(\underline\triangle)$, and from the points on the line $l_\ep$ only
$(\ep-1,\ep)$ and $(\ep,\ep+1)$ are black.
If $\ep=2$, then only $(2,3)$ is black, while only $(e-1,e)$ is
black if $\ep=e$. 
By deleting the base line we get the blow-up $\Bl_\ep({\triangle})$.
In terms of pictures
this means that one moves the sector, bounded by $l_\ep$ and $l_{\ep-1}$
with lowest point $(\ep-1,\ep)$, one position up, and moves the arising
two triangles sideways, to make room for a new line $l_\ep$, which 
has no black dots in $\Bl_\ep({\triangle})$. If $\ep=2$ or $\ep=e$
one just adds an extra line without black dots to the triangle.
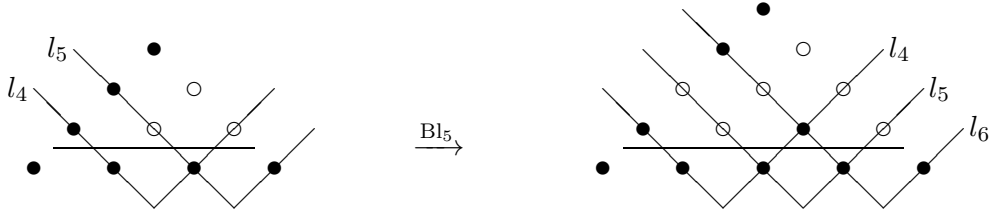
\begin{example}
\[
\begin{picture}(150,60)(-30,-15)
\put(0,0){\zwart}
\put(30,0){\zwart}
\put(60,0){\zwart}
\put(90,0){\zwart}
%\put(-15,-15){\line(1,1){75}}
%\put(-15,-15){\line(-1,1){15}}
%\put(15,-15){\line(1,1){60}}
%\put(15,-15){\line(-1,1){30}}
\put(45,-15){\line(1,1){45}}
\put(45,-15){\line(-1,1){45}}
\put(75,-15){\line(1,1){30}}
\put(75,-15){\line(-1,1){60}}
%\put(105,-15){\line(1,1){15}}
%\put(105,-15){\line(-1,1){75}}
\put(7.5,7.5){\line(1,0){75}}
\put(15,15){\zwart}
\put(45,15){\wit}
\put(75,15){\wit}
\put(30,30){\zwart}
\put(60,30){\wit}
\put(45,45){\zwart}
%\put(28,60){\makebox(0,0)[r]{$l_6$}}
\put(13,45){\makebox(0,0)[r]{$l_5$}}
\put(-2,30){\makebox(0,0)[r]{$l_4$}}
%\put(-17,15){\makebox(0,0)[r]{$l_3$}}
%\put(-32,0){\makebox(0,0)[r]{$l_2$}}
\end{picture}
\qquad
\raisebox{20pt}{$\xrightarrow{\Bl_5}$}
\qquad
\begin{picture}(180,75)(-30,-15)
\put(0,0){\zwart}
\put(30,0){\zwart}
\put(60,0){\zwart}
\put(90,0){\zwart}
\put(120,0){\zwart}
%\put(-15,-15){\line(1,1){90}}
%\put(-15,-15){\line(-1,1){15}}
%\put(15,-15){\line(1,1){75}}
%\put(15,-15){\line(-1,1){30}}
\put(45,-15){\line(1,1){60}}
\put(45,-15){\line(-1,1){45}}
\put(75,-15){\line(1,1){45}}
\put(75,-15){\line(-1,1){60}}
\put(105,-15){\line(1,1){30}}
\put(105,-15){\line(-1,1){75}}
%\put(135,-15){\line(1,1){15}}
%\put(135,-15){\line(-1,1){90}}
\put(7.5,7.5){\line(1,0){105}}
\put(15,15){\zwart}
\put(45,15){\wit}
\put(75,15){\zwart}
\put(105,15){\wit}
\put(30,30){\wit}
\put(60,30){\wit}
\put(90,30){\wit}
\put(45,45){\zwart}
\put(75,45){\wit}
\put(60,60){\zwart}
%\put(62,60){\makebox(0,0)[l]{$l_2$}}
%\put(77,45){\makebox(0,0)[l]{$l_3$}}
\put(107,45){\makebox(0,0)[l]{$l_4$}}
\put(122,30){\makebox(0,0)[l]{$l_5$}}
\put(137,15){\makebox(0,0)[l]{$l_6$}}
\end{picture}
\]
\end{example}
The inverse process is called blowing down at $\ep$. This is
possible at $\ep$, for $2<\ep<e$, if the dot $(\ep-1,\ep+1)$ is
black; by lemma \ref{sparse} the line $l_\ep$ does not contain black dots.
Actually, if $l_\ep$ is empty, but $\al_\ep>1$, i.e., there  are
black dots above it, then it follows that  $(\ep-1,\ep+1)$ is
black: otherwise there cannot be enough black dots in a triangle
with black vertex on the  lowest level. 

\begin{proposition}
The ideal $J'$ has $C_{e-3}=
\frac{1}{e-2} \binom{2(e-3)}{ e-3}$ reduced components.
\end{proposition} 

\begin{proof}
If $\sig\ep11=0$ for all $2<\ep<e-1$, the equations are satisfied:
if a triangle $\triangle$ contains black dots, there has to be at least
one on the base line, at  $(\ep-1,\ep+1)$, so $\sig\ep11=0$ for that $\ep$;
an equation $\la(\triangle)$ for an empty triangle ends with 
a $\sig\ep11$ for some $\ep<e-1$, and $\rh(\triangle)$ starts with
a  $\sig\ep11$ for some $\ep>2$. So the Artin component is a component.

Suppose now that there exists an $\ep$ with $\sig\ep11\neq0$.
Let $J_\ep'$ be the saturation of $J'$ by $\sig\ep11$, i.e., 
$J_\ep'=\cup(J':(\sig\ep11)^i)$. It yields the
equation $l_\ep-r_\ep=0$, so $\sig\ep20=\sig\ep11=\sig\ep02$.
We conclude that $\sig{\ep-1}20=\sig{\ep-1}11=0$ and
$\sig{\ep+1}11=\sig{\ep+1}02=0$ (if $\ep=2$ or $\ep=e-1$ the statements
have to be modified somewhat). As 
$\sig{\ep-1}20-\sig{\ep-1}11=\sig{\ep-1}21(l_{\ep-1}-r_{\ep-1})$,
one has  $\sig{\ep-1}21(l_{\ep-1}-r_{\ep-1})\in J_\ep'$, and
likewise $\sig{\ep+1}12(l_{\ep+1}-r_{\ep+1})\in J_\ep'$.

Consider a monomial $\la(\triangle_{\be,\ga})$ or  
$\rh(\triangle_{\be,\ga})$, containing $\sig\ep11$ (or $\sig\ep20$
if $\be=\ep$, or $\sig\ep02$ if $\ep=\ga$). Then 
$\triangle_{\be,\ga}=\Bl_\ep(\triangle_{\be,\ga-1})$, and the monomial
in question is obtained from $\la(\triangle_{\be,\ga-1})$ or  
$\rh(\triangle_{\be,\ga-1})$ by leaving the $\sig\de ij$ unchanged
for $\de<\ep-1$, replacing $\sig{\ep-1}ij$ by $\sig{\ep-1}{i+1,}j$,
inserting $\sig\ep11$, replacing $\sig\ep ij$ by $\sig{\ep+1}i{,j+1}$
and  $\sig\de ij$ by  $\sig{\de+1}ij$ for $\de>\ep$. We claim that the
polynomials considered so far, together with the monomials, not
involving $\sig{\ep-1}ij$, $\sig{\ep}ij$ and $\sig{\ep+1}ij$ at all,
generate the ideal $J_\ep'$. It follows then that this ideal, 
up to renaming the coordinates as above, and up to some linear equations,
is an ideal of the same type as $J'$, but one embedding dimension lower.
By induction we conclude that $J_\ep'$ describes components,
parametrised by sparse coloured triangles, blown up at $\ep$. By 
varying $\ep$, with $\sig\ep11\neq0$, we obtain all components
(except the Artin component, which we already have).

It remains to prove the claim. The not yet considered generators of $J'$
come in two types, those containing $\sig\ep ij$ with $i+j>2$, and
those not containing a   $\sig\ep ij$ at all, but ending with
$\sig{\ep-1}ij$ or starting with $\sig{\ep+1}ij$. Regarding the first type,
we prove  that such a monomial is a multiple of one of the claimed
generators, by induction on the length of the monomial. For this we note
that the claim holds for the monomial if and only if it holds for the
monomial, obtained by blowing down the triangle at $\de$ with $\de\neq
\ep-1,\ep+1$. The base of the induction  is the case of monomials
containing $\sig{\ep-1}20$,  $\sig{\ep-1}11$, $\sig{\ep+1}11$ 
or $\sig{\ep+1}02$. As to the second type, we consider those
starting with  $\sig{\ep+1}ij$. If the term is of the form
$\la(\triangle{\ep+1,\ga})$, then it starts with $\sig{\ep+1}i0$.
One has $\sig{\ep+1}i0=\sig{\ep+1}{i-1,}1+\sig{\ep+1}i1(l_{\ep+1}-r_{\ep+1})$.
The term obtained by replacing  $\sig{\ep+1}i0$ by $\sig{\ep+1}i1$,
is one of our generators. We are left with monomials, starting
with $\sig{\ep+1}{i-i,}1$; such monomials also come from
$\rh(\triangle{\ep+1,\ga})$. For those the claim is again shown by
induction, using blowing down.
\end{proof}

%=============================================================================

%=============================================================================

\end{document}